\documentclass[12pt,reqno]{amsart}

\usepackage{amsmath,amssymb,amsthm}
\usepackage{bm,bbm}
\usepackage{mathrsfs}
\usepackage[all]{xy}
\usepackage{booktabs}
\usepackage{fullpage}
\usepackage[pagebackref=true]{hyperref}
\usepackage{mathtools}
\usepackage[abbrev,msc-links]{amsrefs}
\usepackage{caption,subcaption}
\usepackage[]{todonotes}
\usepackage{tikz}
\usetikzlibrary{decorations.markings}
\usepackage{pgfplots}
\hypersetup{
 hidelinks,
 bookmarksopen=true,
}

\newcommand{\rl}{\mathbb{R}}

\theoremstyle{plain}
\newtheorem{theorem}{Theorem}[section]

\newtheorem{lemma}[theorem]{Lemma}
\newtheorem{proposition}[theorem]{Proposition}
\newtheorem{corollary}[theorem]{Corollary}
\theoremstyle{definition}
\newtheorem{definition}[theorem]{Definition}

\newtheorem{example}[theorem]{Example}
\theoremstyle{definition}
\newtheorem{remark}[theorem]{Remark}
\newtheorem{remarks}[theorem]{Remarks}

\begin{document}

\title{Magnitude of metric measure spaces and integrals over geodesics}
\author{Yoshinori Hashimoto}
\date{\today}

\begin{abstract}
We propose a definition of magnitude for a length space with a Borel measure, which involves integrals over the set of geodesics. This quantity agrees with the magnitude of finite metric spaces, up to re-scaling the metric to ensure the convergence, when we use the counting measure on them. We also prove a version of the homogeneous magnitude theorem, by showing that the new definition agrees with the volume when we use the weight measure on a compact homogeneous Riemannian manifold. We compute various examples, which suggest that this quantity can capture information of non-uniqueness of geodesics, such as the injectivity radius, corresponding to the generating degrees of the magnitude homology.
\end{abstract}

\maketitle

\tableofcontents

\section{Introduction}

\subsection{Overview of the results}

The magnitude is an invariant of metric spaces introduced by Leinster \cite{Lein13}, extensively studied recently, which is primarily defined for finite metric spaces. There are mainly two approaches (see e.g.~\cite[\S 3]{Lein13} or \cite[\S 2]{Mec13}) to extend this definition to infinite metric spaces: one is to take the supremum of magnitude over all its finite subspaces (or the Gromov--Hausdorff limit of them \cite{Lein13,LeiWil,Willerton09}), the other is to compute the volume with respect to the weight measure as done by Willerton \cite{Willerton}. Meckes \cite[Theorems 2.3 and 2.4]{Mec13} proved that these definitions agree for positive definite metric spaces, such as the ones listed in \cite[Theorem 3.6]{Mec13}, assuming that a weight measure exists.

The aim of this paper is to try to give a unified framework for these two approaches, by proposing a new definition of magnitude for metric measure spaces. Let $(X, \mathsf{d})$ be a length space and $\mu$ be a Borel measure. Integrating over the set $\Omega_{\bm{x}}$ of $n$ geodesic segments connecting $\bm{x} = (x_0 , \dots , x_n) \in X^{n+1}$ with respect to a measure $\Gamma_{\bm{x}}$, we define the magnitude as
\begin{equation*}
		\mathrm{Mag} (X,\mathsf{d},\mu, \Gamma) := \mu (X) + \sum_{n=1}^{\infty} (-1)^n \int_{\bm{x} \in X^{n+1}}  \left(  \int_{\gamma \in \Omega_{\bm{x}}} e^{-\mathrm{Length} (\gamma) }  \mathrm{d} \Gamma_{\bm{x}} \right) \mathrm{d} \mu^{n+1} ,
\end{equation*}
where $\mu^{n+1}$ is the product measure on $X^{n+1}$, assuming that the limit exists (see Definition \ref{dfnptmmgv} for more details). It takes a simplified form, which is valid for a compact Riemannian manifold $(X,g)$ and a probability measure $\mu$ on $X$ that is absolutely continuous with respect to the Lebesgue measure (see Definition \ref{dfnptmg} and Theorem \ref{thgmsnmnpc}), as
\begin{equation*}
	\mathrm{Mag}(X, \mathsf{d}, \mu ) =\mu (X) +\sum_{n=1}^{\infty} (-1)^n \int_{(x_0, \dots , x_n) \in X^{n+1}} e^{-\sum_{k=1}^n \mathsf{d}(x_{k}, x_{k+1})} \mathrm{d} \mu^{n+1},
\end{equation*}
where $\mathsf{d}$ is the length metric associated to $g$.

Its key properties, and the summary of the results in this paper, are the following.
\begin{enumerate}
	\item When $(X , \mathsf{d})$ is a finite metric space and $\mu_{\sharp}$ is the counting measure, $\mathrm{Mag}(X, \mathsf{d}, \mu_{\sharp} )$ agrees with the original definition of magnitude by re-scaling $\mathsf{d}$ to ensure the convergence if necessary (Theorem \ref{lmcvfmsm}). In this formalism, the magnitude of a finite subspace $Z \subset X$ corresponds to choosing $\mu$ to be the counting measure supported on $Z$ (Lemma \ref{lmfsbepm}).
	\item If a weight measure $\mu_{\mathrm{w}}$ exists for a compact Riemannian manifold $(X , g)$ and if it is absolutely continuous with respect to the Lebesgue measure, $\mathrm{Mag}(X, \mathsf{d}, \mu_{\mathrm{w}} )$ is given by the volume $\mu_{\mathrm{w}} (X)$, up to taking a subsequence (Theorem \ref{thmgwgtms}). In particular, $\mathrm{Mag}(X, \mathsf{d}, \mu_{\mathrm{w}} )$ satisfies the homogeneous magnitude theorem as in \cite{Willerton} when $(X,g)$ is a compact homogeneous Riemannian manifold (Corollary \ref{crmgwgtms}).
	\item When $(X,g)$ is a Riemannian manifold and $\mathsf{d}$ is its length metric, it is natural to take $\mu$ to be the volume form $\mathrm{dvol}_g$ of $g$. Examples in \S \ref{scexmfduqg} and \S \ref{scexmfdcl} indicate that $\mathrm{Mag}(X, \mathsf{d}, \mathrm{dvol}_g )$ contains information of the injectivity radius and the diameter of $(X,g)$, which seems to correspond to the non-trivial generators of the magnitude homology group, or the jump in the number of geodesics connecting two points.
	\item In the form that involves integrals over the set of geodesics, we can define new invariants $\mathrm{Mag}(X, \mathsf{d}, \mu , \Gamma )$ by choosing various measures for $\Gamma$ on the set of geodesics (Definition \ref{dfnptmg}). These invariants give different values when geodesics are not unique, as we see for finite graphs (especially 4-cuts) in \S \ref{scitfms}.
\end{enumerate}

Concerning the classical approach of taking the supremum of magnitude over all finite subspaces, the first point above gives a perspective of formulating this problem as approximating a given measure by a counting measure. We can consider a related problem by using an empirical measure (the average of delta measures supported on finite points). Berman, Boucksom, and Witt Nystr\"om \cite{BB10,BBW} proved that the empirical measure on the Fekete configuration converges to the equilibrium measure on a compact K\"ahler manifold. We can prove, by using this result, that the magnitude $\mathrm{Mag}(X, \mathsf{d}, \mu )$ computed for an equilibrium measure $\mu$ is a limit of an appropriately scaled version of magnitude of the Fekete configurations (Theorem \ref{thmfkeqm}). For $\mathbb{CP}^1$, we can also establish a relationship between the classical magnitude and the limit of the magnitude of Fekete configurations (Corollary \ref{thmclmfkc1}).

The second point above shows that $\mathrm{Mag}(X, \mathsf{d}, \mu_{\mathrm{w}} )$ agrees with the classical magnitude for spheres, which is positive definite, when we use the weight measure $\mu_{\mathrm{w}}$ (see also Corollary \ref{clmwsmmz}). We also give an example (see \S \ref{sccbism}) for which the new definition $\mathrm{Mag}(X, \mathsf{d}, \mu )$ does not agree with the classical one when there is a singular component in the weight measure.

The third point above indicates that the invariant $\mathrm{Mag}(X, \mathsf{d}, \mu )$ may be related to the de-categorification of the magnitude homology, at least for the examples covered in this paper, which include $\mathbb{R}$, $S^1$, $S^2$, and a 2-torus (see \S \ref{scexmfduqg} and \S \ref{scexmfdcl}). Recall that the magnitude homology measures the non-uniqueness of geodesics, as proved by Asao \cite{Asao} and Gomi \cite{Gomi}. The examples in \S \ref{scexmfdcl} show that the new definition of magnitude function can detect the injectivity radius or the diameter, which is deeply connected to the non-uniqueness of geodesics, as the rate of exponential decay. These are the length parameters for which the magnitude homology is non-trivial, as computed by Gomi \cite[Theorem 1.3]{Gomi}. When the manifold is uniquely geodesic, such as $(\mathbb{R}, | \cdot |)$ treated in \S \ref{scexmfduqg}, there seem to be no terms of such decay (i.e.~no terms of order $e^{-ct}$ for some constant $c>0$). On the other hand, the example in \S \ref{sccbilm} shows that this naive picture overlooks the contributions from the boundary.

The fourth point above means the variety of definitions that we have, by choosing various measures $\mu$ and $\Gamma$.

\subsection{Organisation of the paper}

We review the classical definition of magnitude for finite metric spaces in \S \ref{scrrirffms}, with a path-integral like interpretation to motivate the main part of the text. The definitions of the invariants $\mathrm{Mag} (X,\mathsf{d},\mu, \Gamma)$, $\mathrm{Mag} (X,\mathsf{d},\mu)$ are given in \S \ref{scpidmmms}, together with some sufficient conditions to ensure their convergence, and also a relationship to the Fekete configurations. In \S \ref{scrm}, we prove some preliminary results that apply for Riemannian manifolds with weighted volume forms, before computing various examples in \S \ref{scexmfduqg} and \S \ref{scexmfdcl}. Finally, we treat the case when the metric space admits a positive weight measure in \S \ref{scmswms}, and prove various results including the homogeneous magnitude theorem for compact homogeneous Riemannian manifolds.

\subsection{Notational conventions and terminologies}
\begin{itemize}
	\item Throughout this paper, a metric space $(X , \mathsf{d})$ is either a finite metric space or a length space. For a Borel measure $\mu$ on $(X , \mathsf{d})$, we write $\mu^n$ for the product measure on $X^n$.
	\item When $X$ is finite, $\mu_{\sharp}$ denotes the counting measure, with the defining property $\mu_{\sharp} (A) = |A|$ for any $A \subset X$, where we write $|A|$ for the cardinality of $A$.
	\item For a Riemannian manifold $(X,g)$, we write $\mathsf{d}_g$ for the length metric with respect $g$, $\mathrm{dvol}_g$ for the volume form, and $\mathrm{Vol}_g(X):= \int_X \mathrm{dvol}_g$ for the total volume.
	\item For a square matrix $Y$, we write $\Vert Y \Vert_{\mathrm{op}}$ for the operator norm.
\end{itemize}

\medskip

\noindent \textbf{Acknowledgements.} The author thanks Yasuhiko Asao for many helpful comments and for patiently answering questions during discussions, and Takahiro Aoi, Kiyoon Eum, Yuichi Nohara for helpful comments. He also thanks the organisers and the participants of the workshops Magnitude 2024 and 2025 for providing him with immensely helpful learning opportunities. This work is partially supported by JSPS KAKENHI Grant Number JP23K03120, JP24K00524, and by MEXT Promotion of Distinctive Joint Research Center Program JPMXP0723833165 and Osaka Metropolitan University Strategic Research Promotion Project (Development of International Research Hubs).

\section{Review and re-interpretation of results for finite metric spaces} \label{scrrirffms}
Let $(X,\mathsf{d})$ be a finite metric space and $q$ be a formal variable. We define a square matrix $Z_X$ of size $|X|$ by
\begin{equation} \label{eqdfsmlmx}
	(Z_X)_{xy} := q^{\mathsf{d}(x,y)}
\end{equation}
for $x,y \in X$, and set
\begin{equation} \label{eqtrdmag}
	\mathrm{Mag}(X, \mathsf{d} ):= \sum_{x,y \in X} (Z^{-1}_X)_{xy} \in \mathbb{Z} [ \! [ q^{>0} ] \! ]
\end{equation}
where $\mathbb{Z} [ \! [ q^{>0} ] \! ]$ stands for the formal power series ring generated by positive powers of a formal variable $q$, noting that $\det Z_X$ is invertible in $\mathbb{Z} [ \! [ q^{>0} ] \! ]$. Note that a geometric interpretation of magnitude is also given in recent works \cite{AG25,Dev25}. Define another square matrix $Y_X$ of size $|X|$ as $Y_X := Z_X -I_{|X|}$, which satisfies
\begin{equation} \label{eqdfyxxy}
	(Y_X)_{xy} = (Z_X-I_{|X|})_{xy} = \begin{cases}
		0 &\quad (x=y) \\
		q^{\mathsf{d}(x,y)} &\quad (x \neq y)
	\end{cases}
\end{equation}
to write, following \cite[Proof of Theorem 7.14]{LeiShu21},
\begin{align}
	(Z_X^{-1})_{xy} &= ((I_{|X|}+Y_X)^{-1})_{xy} \notag \\
	&= \sum_{n=0}^{\infty} (-1)^k (Y^k_X)_{xy} \label{exyxivm} \\
	&= \delta_{xy}+\sum_{n=1}^{\infty} \sum_{\substack{x=x_0 \neq \cdots \neq x_{n}=y \\ x_2 , \dots , x_{n} \in X}} (-1)^n q^{\mathsf{d}(x_1, x_2) + \mathsf{d}(x_2, x_3) + \cdots + \mathsf{d}(x_{n-1}, x_{n})}, \label{eqzxivexp}
\end{align}
where $\delta_{xy}$ is the Kronecker delta, which makes sense as an element in $\mathbb{Z} [ \! [ q^{>0} ] \! ]$ (or in fact in $\mathbb{Z} [ \! [ q ] \! ]$ if the distance is integer-valued, as in graphs). Recall that a collection of finitely many ordered points $x_0 , x_1 , \dots , x_n$ in $X$ is said to be a \textbf{proper chain} if $x_{k-1} \neq x_k$ for all $k=1 , \dots , n$ (see e.g.~\cite[\S 2.1]{KY21}). The set of all proper chains is commonly denoted by
\begin{equation} \label{eqdfnppch}
	P_n (X) := \{ (x_0, x_1 , \dots , x_{n}) \in X^{n+1} \mid x_{k-1} \neq x_k \text{ for all } k=1 , \dots ,  n \} .
\end{equation}
By using the counting measure $\mu_{\sharp}$ on $X$, and its product measure $\mu_{\sharp}^{n+1}$ on $X^{n+1}$, we can write the magnitude as
\begin{equation*}
	\mathrm{Mag}(X, \mathsf{d} ) = \mu_{\sharp} (X) +\sum_{n=1}^{\infty} (-1)^n \int_{P_n (X)} q^{\sum_{k=1}^n \mathsf{d} (x_{k-1} , x_k)} \mathrm{d} \mu_{\sharp}^{n+1}
\end{equation*}
by writing the sum over $x,y$ in (\ref{eqzxivexp}) in terms of the counting measure. Note that essentially the same formula was presented in \cite[equation (1)]{HepWil17}, \cite[Proposition 3.9]{Lein19}, \cite[Corollary 7.15]{LeiShu21}, and \cite[Proposition 2.1]{Ohara24}.

While the above equation can serve as a starting point of most of what follows, computation similar to the equation (\ref{eqzxivexp}) is also well-known in the context of path integrals (see e.g.~\cite{FH,Mum}), which gives further flexibility in the definition of the magnitude in \S \ref{scpidmmms}. By hypothetically interpreting each pair of distinct points $x_{k-1}, x_k \in X$ to be connected by a unique geodesic of length $\mathsf{d} (x_{k-1},x_k)$ for each $k=1, \dots , n$, we also call a proper chain $\gamma := ( x_0 , x_1 , \dots , x_n ) \in X^{n+1}$ to be an \textbf{$n$-piecewise geodesic} connecting $x_0,x_n \in X$; we just changed the terminology with no mathematical content, to motivate the later definition in \S \ref{scpidmmms}, and the existence and uniqueness of these (hypothetical) geodesics are purely for the sake of exposition (see in particular \S \ref{scfms}). For an $n$-piecewise geodesic $\gamma := ( x_0 , x_1 , \dots , x_n )$, its \textbf{length} can be naturally defined as
\begin{equation*}
	\mathrm{Len} (\gamma) := \sum_{k=1}^n \mathsf{d} (x_{k-1} , x_k).
\end{equation*}
Writing
\begin{equation*}
	\Omega^n_{x,y} := \{ \text{$n$-piecewise geodesics connecting $x$ and $y$}\},
\end{equation*}
the formula (\ref{eqzxivexp}) can be interpreted as
\begin{equation*}
	(Z_X^{-1})_{xy} = \delta_{xy}+\sum_{n=1}^{\infty} \sum_{\gamma \in \Omega^n_{x,y}} (-1)^{n}q^{\mathrm{Len} (\gamma)} \in \mathbb{Z} [ \! [ q^{>0} ] \! ] .
\end{equation*}
Introducing a discrete counting measure $\Gamma^{n, \sharp }_{x,y}$ on each finite set $\Omega^n_{x,y}$, we may write the above as
\begin{equation*}
	(Z_X^{-1})_{xy} = \delta_{xy}+\sum_{n=1}^{\infty} (-1)^n \int_{\gamma \in \Omega^n_{x,y}} q^{\mathrm{Len} (\gamma)} \mathrm{d} \Gamma^{n, \sharp }_{x,y} . 
\end{equation*}
Thus, with the counting measure $\mu_{\sharp}$ on $X$, we get the ``path integral'' formalism of the magnitude as
\begin{equation*}
\mathrm{Mag}(X , \mathsf{d}) = \mu_{\sharp} (X) +\sum_{n=1}^{\infty} (-1)^n \int_{X \times X} \left( \int_{\gamma \in \Omega^n_{x,y}} q^{\mathrm{Len} (\gamma)} \mathrm{d} \Gamma^{n, \sharp }_{x,y} \right) \mathrm{d} \mu_{\sharp}^2 , 
\end{equation*}
where we decreed $\Omega^1_{x,y} = \emptyset$ when $x=y$ (this is consistent with the usual definition of geodesics, as in \cite{BriHae}).

A point $\bm{x}:= (x_0, x_1 , \dots , x_{n}) \in P_n (X)$ corresponds one-to-one with an $n$-piecewise geodesic $\gamma$ connecting $x_0$ and $x_n$. We write $\Omega_{\bm{x}}$ for the singleton set consisting of the unique $\gamma$ corresponding to $\bm{x}= (x_0, x_1 , \dots , x_{n}) \in P_n (X)$, endowed with the counting measure $\Gamma_{\bm{x}}$ which is necessarily trivial, in the sense that $\Gamma_{\bm{x}} (\Omega_{\bm{x}}) = 1$ for any $\bm{x} \in P_n (X)$.

The above argument gives us the following re-formulation of the magnitude, which naturally involves a counting measure in addition to the finite metric space $(X,\mathsf{d})$.

\begin{proposition} \label{ppfmgpitg}
	Let $(X,\mathsf{d})$ be a finite metric space, and $\mu_{\sharp}$ be a counting measure on $X$. Then the magnitude can be written as
	\begin{equation*}
		\mathrm{Mag}(X, \mathsf{d} ) = \mu_{\sharp} (X) +\sum_{n=1}^{\infty} (-1)^n \int_{P_n (X)} q^{\sum_{k=1}^n \mathsf{d} (x_{k-1} , x_k)} \mathrm{d} \mu_{\sharp}^{n+1} \in \mathbb{Z} [ \! [ q^{>0} ] \! ].
	\end{equation*}
	Moreover, if we decree that any two distinct points in $X$ can be connected by a unique geodesic, the magnitude of $X$ can be written as
	\begin{equation*}
	\mathrm{Mag}(X , \mathsf{d}) =\mu_{\sharp} (X) +  \sum_{n=1}^{\infty} (-1)^n \int_{\bm{x} \in X^{n+1}}  \left(  \int_{\gamma \in \Omega_{\bm{x}}} q^{\mathrm{Len} (\gamma)}  \mathrm{d} \Gamma_{\bm{x}} \right) \mathrm{d} \mu^{n+1}_{\sharp} \in \mathbb{Z} [ \! [ q^{>0} ] \! ] .
	\end{equation*}
\end{proposition}

We further define
\begin{equation*}
	P^l_n (X):= \left\{ (x_0, \dots , x_n ) \in P_n (X) \; \left| \; \sum_{k=1}^n \mathsf{d} (x_{k-1}, x_k)= l \right\} \right.
\end{equation*}
to write
\begin{align}
	\mathrm{Mag}(X , \mathsf{d}) &= |X| + \sum_{n=1}^{\infty} \sum_{l>0} (-1)^n\int_{\bm{x} \in P^l_n (X)}  \left(  \int_{\gamma \in \Omega_{\bm{x}}} q^{\mathrm{Len} (\gamma)}  \mathrm{d} \Gamma_{\bm{x}} \right) \mathrm{d} \mu^n_l \notag \\
	&=|X| +\sum_{n=1}^{\infty} (-1)^n \sum_{l>0} q^l\int_{\bm{x} \in P^l_n (X)} |\Omega_{\bm{x}}| \mathrm{d} \mu^n_l \label{eqcmpilsh}
\end{align}
where $\mu^n_l$ is the counting measure on $P^l_n (X)$. Since $\Omega_{\bm{x}}$ is a singleton set, we recover the well-known formula proved in \cite[equation (1)]{HepWil17}, \cite[Proposition 3.9]{Lein19}, \cite[Corollary 7.15]{LeiShu21}
\begin{equation*}
	\mathrm{Mag}(X , \mathsf{d}) =|X| +\sum_{n=1}^{\infty} (-1)^n \sum_{l=1}^{\infty} q^l | P^l_n (X) |
\end{equation*}
which is the Euler characteristic of the magnitude chain complex added over all $l >0$, by noting that $P^l_n (X)$ is precisely the generator of the magnitude chain complex $MC^l_n (X)$, whence $\mathrm{rk} MC^l_n (X) = | P^l_n (X) |$.

In what follows, we shall take
\begin{equation*}
	q:=e^{-1}.
\end{equation*}
The infinite sums above may not converge for this choice of $q$, but we may replace it by a power $e^{-t}$ for some $t >0$ to guarantee the convergence, which is well-known to be equivalent to scaling the metric $\mathsf{d}$ by $t$. It suffices to take $t$ to be large enough so that the expansion (\ref{exyxivm}) converges. This condition is satisfied, for example, when the operator norm of $Y_{X,t}$ satisfies $\Vert Y_{X,t} \Vert_{\mathrm{op}} <1$, where $Y_{X,t}$ is a matrix defined similarly to (\ref{eqdfyxxy}) as
\begin{equation}
	(Y_{X,t})_{xy} := \begin{cases}
		0 &\quad (x=y) \\
		e^{-t \mathsf{d}(x,y)} &\quad (x \neq y) .
	\end{cases}
\end{equation}
The condition $\Vert Y_{X,t} \Vert_{\mathrm{op}} <1$ is satisfied when
\begin{equation} \label{eqprfrsrd}
	\max_{y \in X} \sum_{x \in X} e^{-t \mathsf{d}(x,y)} <1
\end{equation}
by \cite[Lemma 8.1.21]{HJ13}.

\begin{remark}
	In the interpretation above, the magnitude weighting $v_x$ can be written as
\begin{equation*}
	v_x = 1+\sum_{n=1}^{\infty} (-1)^n \int_{y \in X} \left( \int_{\gamma \in \Omega^n_{x,y}} q^{\mathrm{Len} (\gamma)} \mathrm{d} \Gamma^{n, \sharp }_{x,y} \right) \mathrm{d} \mu_{\sharp}^2 \in \mathbb{Z} [ \! [ q^{>0} ] \! ] 
\end{equation*}
for each $x \in X$, and hence can be interpreted as a density function for the magnitude.
\end{remark}

\section{Magnitude of metric measure spaces} \label{scpidmmms}

When $(X, \mathsf{d})$ is an infinite metric space, the magnitude is traditionally defined as the supremum
\begin{equation} \label{eqdfclmg}
	\mathrm{Mag}(X , \mathsf{d}):= \sup_{Y \subset X, \; \text{finite}} \mathrm{Mag} (Y, \mathsf{d} |_Y)
\end{equation}
over all finite subsets of $X$. While a lot is known about the magnitude defined as above, we propose in this paper an alternative definition following the formulation in Proposition \ref{ppfmgpitg}, which admits a straightforward generalisation even when $X$ is not finite, at least formally, by replacing the counting measure by a general one. Two novel features are that it involves choosing the measure on $X$ and on the space of geodesics.

 To avoid confusion, we call the original definitions (\ref{eqtrdmag}) and (\ref{eqdfclmg}) the \textbf{classical magnitude} in this paper.

\subsection{General definition}

Let $(X,\mathsf{d})$ be a length space, and $\mu$ be a Borel measure on $(X,\mathsf{d})$; in particular, the triple $(X,\mathsf{d}, \mu)$ is a metric measure space. We gave two formulations of the magnitude for finite metric spaces in Proposition \ref{ppfmgpitg}. The first one can be generalised as follows.

\begin{definition} \label{dfnptmmgv}
	Let $(X,\mathsf{d})$ be a length space, $\mu$ be a Borel measure on $(X,\mathsf{d})$ such that the total volume $\mu (X)$ is finite, and $N \in \mathbb{Z}_{>0}$. The \textbf{$N$-th partial $\mu$-magnitude} of $(X, \mathsf{d})$ is defined as
	\begin{equation*}
		\mathrm{Mag} (X,\mathsf{d},\mu ; N) := \mu (X) + \sum_{n=1}^{N} (-1)^n \int_{\bm{x} \in P_n (X)}  e^{-\sum_{k=1}^n \mathsf{d} (x_{k-1} , x_k) }  \mathrm{d} \mu^{n+1} ,
	\end{equation*}
	where $P_n (X)$ is the set of all proper chains as defined in (\ref{eqdfnppch}). When the limit
\begin{equation*}
\mathrm{Mag} (X,\mathsf{d},\mu ) := \lim_{N \to \infty} \mathrm{Mag} (X,\mathsf{d},\mu ; N )
\end{equation*}
exists in $\mathbb{R} \cup \{ \pm \infty \}$, we call it the \textbf{$\mu$-magnitude} of $(X,\mathsf{d})$.
\end{definition}

Generalisation of the second formalism in Proposition \ref{ppfmgpitg} allows for more flexibility, since there is an extra freedom in choosing a measure $\Gamma$ on the space of $n$-piecewise geodesics. 

\begin{definition} 
	A path $\gamma$ connecting $x,y \in X$ is a continuous map $\gamma : [0,1] \to (X,\mathsf{d})$ such that $\gamma (0) = x$ and $\gamma (1)=y$.  For each proper chain $\bm{x}:= (x_0, x_1 , \dots , x_{n}) \in P_n (X)$, we define \textbf{the space of $n$-piecewise geodesics} $\Omega_{\bm{x}}$ to be the set consisting of rectifiable paths $\gamma$ connecting $x_0$ and $x_n$ and passing through $x_1 , \dots , x_{n-1}$, such that $\gamma$ is a length-minimising geodesic between $x_{k-1}$ and $x_k$ for all $k=1 , \dots , n$.
\end{definition}

In particular, for each $\gamma \in \Omega_{\bm{x}}$, we have
	\begin{equation*}
	\mathrm{Len} (\gamma ) = \sum_{k=1}^n \mathsf{d} (x_{k-1} , x_k),	 \quad \text{and} \quad \mathrm{Len} (\gamma; x_{k-1},x_k) = \mathsf{d} (x_{k-1} , x_k),	
	\end{equation*}
where $\mathrm{Len} (\gamma; x_{k-1},x_k)$ is the length of $\gamma$ between $x_{k-1}$ and $x_k$.

\begin{definition}
	We write $\Omega_{x,y}$ for the set of all non-constant length-minimising geodesics connecting $x , y \in X$ (thus $\Omega_{x,y} = \emptyset$ when $x=y$). We may thus write
\begin{equation*}
	\Omega_{\bm{x}} = \Omega_{x_0 , x_1} \times \cdots \times \Omega_{x_{n-1},x_n}
\end{equation*}
for each proper chain $\bm{x}:= (x_0, x_1 , \dots , x_{n}) \in P_n (X)$. For each $x,y \in X$, we define a measure $\Gamma_{x,y}$ on $\Omega_{x,y}$. In this paper we only consider the following choices for $\Gamma$.
\begin{enumerate}
	\item $\Gamma^{\mathrm{triv}}$: for any $x,y \in X$ and $\Omega_{x,y}$, the \textbf{trivial measure} $\Gamma_{x,y}^{\mathrm{triv}}$ is defined with respect to the trivial $\sigma$-algebra $\{ \emptyset , \Omega_{x,y} \}$ on $\Omega_{x,y}$, such that $\Gamma_{x,y}^{\mathrm{triv}} ( \Omega_{x,y} ) =1$ if $\Omega_{x,y} \neq \emptyset$.
	\item $\Gamma^{\sharp}$: when $\Omega_{x,y}$ is finite for $x,y \in X$, we can also define the \textbf{counting measure} $\Gamma_{x,y}^{\sharp}$ with respect to the $\sigma$-algebra given by the power set of $\Omega_{x,y}$, such that $\Gamma_{x,y}^{\sharp} ( A ) = |A|$ for any $A \subset \Omega_{x,y}$.
\end{enumerate}
For either choice of $\Gamma$, $\Omega_{\bm{x}}$ has a natural product measure $\Gamma_{\bm{x}}$ induced from $\Gamma_{x_0, x_1}, \dots , \Gamma_{x_{n-1}, x_n}$. We then have
\begin{equation} \label{eqdfomgmk}
	\int_{\gamma \in \Omega_{\bm{x}}} e^{-\mathrm{Len} (\gamma)}  \mathrm{d} \Gamma_{\bm{x}} = e^{-\sum_{k=1}^n \mathsf{d}(x_{k-1}, x_k)} \prod_{k=0}^n \int_{\Omega_{x_{k-1}, x_k}} \mathrm{d}\Gamma_{x_{k-1}, x_k}.
\end{equation}
\end{definition}

\begin{definition} \label{dfnptmg}
	Let $(X,\mathsf{d})$ be a length space, and $\mu$ be a Borel measure on $(X,\mathsf{d})$ such that the total volume $\mu (X)$ is finite. We choose a cutoff parameter $N \in \mathbb{Z}_{>0}$, and a family of measures $\{ \Gamma_{\bm{x}} \}_{\bm{x} \in P_n(X)}$ on the spaces of $n$-piecewise geodesics $\{ \Omega_{\bm{x}} \}_{\bm{x} \in P_n (X)}$, such that the function
	\begin{equation*}
		P_n (X) \ni \bm{x} \mapsto \int_{\gamma \in \Omega_{\bm{x}}} e^{-\mathrm{Len} (\gamma)}  \mathrm{d} \Gamma_{\bm{x}} \in \mathbb{R} , 
	\end{equation*}
	as defined in (\ref{eqdfomgmk}) is integrable over $P_n (X)$ with respect to the product measure $\mu^{n+1}$, for all $n =1, \dots , N$. Then, the \textbf{$N$-th partial $(\mu, \Gamma )$-magnitude} of $(X, \mathsf{d})$ is defined as
	\begin{equation*}
		\mathrm{Mag} (X,\mathsf{d},\mu, \Gamma ; N) := \mu (X) + \sum_{n=1}^{N} (-1)^n \int_{\bm{x} \in P_n (X)}  \left(  \int_{\gamma \in \Omega_{\bm{x}}} e^{-\mathrm{Len} (\gamma) }  \mathrm{d} \Gamma_{\bm{x}} \right) \mathrm{d} \mu^{n+1} .
	\end{equation*}
When the limit
\begin{equation*}
\mathrm{Mag} (X,\mathsf{d},\mu, \Gamma ) := \lim_{N \to \infty} \mathrm{Mag} (X,\mathsf{d},\mu, \Gamma ; N )
\end{equation*}
exists in $\mathbb{R} \cup \{ \pm \infty \}$, we call it the \textbf{$(\mu, \Gamma )$-magnitude} of $(X,\mathsf{d})$.
\end{definition}

\begin{remarks}
	The following observations immediately follow from the definition above.
	\begin{enumerate}
		\item The leading term ($n=0$) is always the volume $\mu (X)$. Thus, the $(\mu , \Gamma)$-magnitude can be regarded as the $\mu$-volume of $X$ with infinitely many corrections given by the metric structure of $X$.
		\item  Note that
		\begin{equation*}
		\int_{\Omega_{x_{k-1}, x_k}} \mathrm{d}\Gamma_{x_{k-1}, x_k} = \Gamma_{x_{k-1}, x_k} (\Omega_{x_{k-1}, x_k} )  =0	
		\end{equation*}
		when there is no geodesic connecting $x_{k-1}$ and $x_k$, in which case the range of integration can be reduced to the subset of $P_n (X)$ that can be pairwise connected by geodesics. In particular, $\Gamma^{\mathrm{triv}}_{\bm{x}} ( \Omega_{\bm{x}})$ can be regarded as an indicator function, which is zero when there exists a pair in $\bm{x} = (x_0, \dots , x_n)$ that cannot be connected by non-constant geodesics, and $1$ otherwise. This observation implies that we have
		\begin{equation*}
		\int_{\bm{x} \in X^{n+1}}  \left(  \int_{\gamma \in \Omega_{\bm{x}}} e^{-\mathrm{Len} (\gamma) }  \mathrm{d} \Gamma_{\bm{x}} \right) \mathrm{d} \mu^{n+1} = \int_{\bm{x} \in P_n (X)}  \left(  \int_{\gamma \in \Omega_{\bm{x}}} e^{-\mathrm{Len} (\gamma) }  \mathrm{d} \Gamma_{\bm{x}} \right) \mathrm{d} \mu^{n+1}
		\end{equation*}
		noting $\Gamma_{x,y} (\Omega_{x,y})= 0 $ when $x=y$, since $\Omega_{x,y} = \emptyset$ in this case. Thus, the range of integration in the definition of the $(\mu , \Gamma)$-magnitude can be extended to $X^{n+1}$, where we note that $\int_{\gamma \in \Omega_{\bm{x}}} e^{-\mathrm{Len} (\gamma) }  \mathrm{d} \Gamma_{\bm{x}}$ jumps discontinuously to zero at $\bm{x} \in X^{n+1} \setminus P_n (X)$.
		\item While any choice of the measure $\Gamma$ works as long as it satisfies the conditions in Definition \ref{dfnptmg}, we mostly deal with $\Gamma^{\mathrm{triv}}$, and occasionally $\Gamma^{\sharp}$, in this paper.
		\item Exactly as in the case when $X$ is finite, we can decree the magnitude weighting to be the density function of the magnitude, but its geometric meaning and its relationship to the magnitude homology is much less clear than the case of finite metric spaces.
	\end{enumerate}
\end{remarks}

The conditions in Definition \ref{dfnptmg} are satisfied when $(X,\mathsf{d})$ is uniquely geodesic (i.e.~any two points can be connected by a unique geodesic, such as CAT(0) spaces, see e.g.~\cite{BriHae}), for which $\Omega_{\bm{x}}$ is a singleton for any $\bm{x} \in  P_n (X)$ and any $n \in \mathbb{Z}_{>0}$. In this case, we have the following formula.

\begin{lemma}
	Suppose that $(X,\mathsf{d})$ is uniquely geodesic, endowed with a measure $\mu$ such that $\mu (X)$ is finite. Then, for any measure $\Gamma$ which assigns $1$ to a singleton set, we have
	\begin{equation*}
		\mathrm{Mag} (X,\mathsf{d},\mu, \Gamma ; N ) = \mu(X)+ \sum_{n=1}^{N} (-1)^n \int_{\bm{x} \in P_n (X)}  e^{- \sum_{k=1}^n \mathsf{d} (x_{k-1}, x_k)} \mathrm{d} \mu^{n+1} .
	\end{equation*}
	for any $N \in \mathbb{Z}_{>0}$. In particular,
	\begin{equation*}
		\mathrm{Mag} (X,\mathsf{d},\mu, \Gamma^{\mathrm{triv}} ; N ) = \mathrm{Mag} (X,\mathsf{d},\mu, \Gamma^{\sharp} ; N ) = \mathrm{Mag} (X,\mathsf{d},\mu; N )
	\end{equation*}
	for any $N \in \mathbb{Z}_{>0}$.
\end{lemma}

Throughout this paper, we will be mostly working with the measure $\Gamma^{\mathrm{triv}}$. While it may seem that $\Gamma^{\mathrm{triv}}$ discards many interesting phenomena to do with non-uniqueness of geodesics, we provide various examples to show that the $(\mu , \Gamma^{\mathrm{triv}})$-magnitude can capture the injectivity radius of a Riemannian manifold, which is closely related to the uniqueness of geodesics. In fact, for most of the examples we deal with in this paper, the $(\mu , \Gamma^{\mathrm{triv}})$-magnitude agrees with the $\mu$-magnitude. In this paper, the choice of $\Gamma$ does not matter as long as it assigns $1$ to a singleton set, except for Examples \ref{ex4cut} and \ref{ex4cutwgt}. Thus, Definition \ref{dfnptmmgv} suffices for most examples; the motivation for giving Definition \ref{dfnptmg} using integrals over geodesics is simply to have more flexibility in the definition.

Since the $(\mu , \Gamma )$-magnitude is given by an infinite series, it is difficult to compute it precisely.  An exception is when $\mu$ is a weight measure on a compact homogeneous Riemannian manifold $X$, for which the $(\mu , \Gamma^{\mathrm{triv}} )$-magnitude reduces to the volume $\mu (X)$ (see Theorem \ref{thmgwgtms} and Corollary \ref{crmgwgtms}). It is even harder to compute the magnitude function $\mathrm{Mag} (X, t \mathsf{d},\mu, \Gamma )$ for $t >0$; it is possible to have a formula for the case when $X$ is a closed interval in $\mathbb{R}$ and $\mu$ is a Lebesgue measure, but it is defined only inductively (see Lemma \ref{lmmgcbinv}). The case when $X$ is a circle with a natural Riemannian volume form $\mu$ already seems very difficult, and the computation is given only up to $N=3$. The details of these examples are given in \S \ref{scexmfdcl}.

\subsection{Finite metric spaces} \label{scfms}

We first prove that Definitions \ref{dfnptmmgv} and \ref{dfnptmg} are consistent with the original one (\ref{eqtrdmag}) for finite metric spaces. While it is straightforward for Definition \ref{dfnptmmgv}, we need to be careful about Definition \ref{dfnptmg}, since any continuous map $\gamma : [0,1] \to (X,\mathsf{d})$ is trivial for finite metric spaces; $(X, \mathsf{d})$ has discrete topology and hence $\gamma$ must be constant.

We decree, as we did in \S \ref{scrrirffms}, that any two distinct points in a finite metric space $(X, \mathsf{d})$ can be connected by a unique length-minimising geodesic, or equivalently that $\Omega_{\bm{x}}$ is a singleton set for any proper chain $\bm{x} \in P_n (X)$. This convention is just a matter of interpretation, and it can be modified for graphs (see Lemma \ref{lmftgrmg}, Examples \ref{ex4cut} and \ref{ex4cutwgt}). In any case, we do not logically rely on this hypothesis outside this subsection, particularly when we deal with the $\mu$-magnitude. The next result shows that, with the interpretation of geodesics as above, Definition \ref{dfnptmg} agrees with the classical magnitude for finite metric spaces.

\begin{theorem} \label{lmcvfmsm}
	Let $(X, \mathsf{d})$ be a finite metric space and $\mu_{\sharp}$ be the counting measure. Then the following hold.
	\begin{enumerate}
		\item If the $\mu_{\sharp}$-magnitude of $(X, \mathsf{d})$ exists, we have
		\begin{equation*}
		\mathrm{Mag} (X,\mathsf{d},\mu_{\sharp}, \Gamma^{\mathrm{triv}}) = \mathrm{Mag} (X,\mathsf{d},\mu_{\sharp}, \Gamma^{\sharp} ) = \mathrm{Mag} (X,\mathsf{d},\mu_{\sharp}),
		\end{equation*}
		when we interpret the $n$-piecewise geodesics for finite metric spaces as above. 
		\item If the power series expansion \textup{(\ref{eqzxivexp})} converges absolutely for $q = e^{-1}$ and agrees with $(Z^{-1}_X)_{x,y}$, for all $x,y \in X$, then the $\mu_{\sharp}$-magnitude of $(X,\mathsf{d})$ exists and agrees with the classical magnitude.
		\item There exists $t_0 >0$ such that the $\mu_{\sharp}$-magnitude of $(X,t \mathsf{d})$ exists and agrees with the classical magnitude of $(X,t \mathsf{d})$ for all $t \ge t_0$.
	\end{enumerate}
\end{theorem}

\begin{proof}
	Noting that the set $\Omega_{\bm{x}}$ is singleton for any $\bm{x} = (x_0 , \dots , x_n ) \in P_n (X)$ and any $n \in \mathbb{Z}_{>0}$, we find
	\begin{equation*}
		\int_{\gamma \in \Omega_{\bm{x}}} e^{-t\mathrm{Len} (\gamma) }  \mathrm{d} \Gamma^{\mathrm{triv}}_{\bm{x}} = \int_{\gamma \in \Omega_{\bm{x}}} e^{-t\mathrm{Len} (\gamma) }  \mathrm{d} \Gamma^{\sharp}_{\bm{x}} = e^{- t \sum_{k=1}^n d(x_{k-1}, x_k)}
	\end{equation*}
	for any $t>0$, which immediately implies the first claim. The second claim follows from the first and Proposition \ref{ppfmgpitg}.

	It thus remains to prove the third claim. Since $X$ is finite, there exists $\epsilon := \min \{  \mathsf{d} (x,y) \mid x,y \in X , x \neq y \}>0$ which satisfies
	\begin{equation*}
		\int_{\bm{x} \in P_n (X)}  e^{- t \sum_{k=1}^n d(x_{k-1}, x_k)} \mathrm{d} \mu_{\sharp}^{n+1} \le e^{- t n \epsilon } |P_n (X)| \le e^{- t n \epsilon } |X^{n+1}| = |X|(e^{-t \epsilon} |X|)^{n}.
	\end{equation*}
	Thus the infinite sum defining $\mathrm{Mag} (X, t \mathsf{d},\mu_{\sharp} )$ is absolutely convergent for $t > (\log |X|) / \epsilon$.
\end{proof}

The magnitude defined as above behaves well with respect to restriction to finite subsets, in the following sense.

\begin{lemma} \label{lmfsbepm}
	Let $(X,\mathsf{d})$ be a geodesic metric space, and let $Z$ be a finite subset of $X$, regarded as a metric space by the restriction of $\mathsf{d}$. We define the counting measure on $X$ supported on $Z$, by using the delta measure as
	\begin{equation*}
		\mu_{Z,\sharp} :=\sum_{z \in Z} \delta_z .
	\end{equation*}
	Then, for each $N \in \mathbb{Z}_{>0}$, we have
	\begin{equation*}
		\mathrm{Mag} (X,\mathsf{d},\mu_{Z,\sharp}, \Gamma^{\mathrm{triv}} ; N) = \mathrm{Mag} (Z, \mathsf{d}|_{Z},\mu_{Z,\sharp}, \Gamma^{\mathrm{triv}} ; N) ,
	\end{equation*}
	which converges as $N \to \infty$ by scaling up $\mathsf{d}$ if necessary, as in Theorem \ref{lmcvfmsm}.
\end{lemma}

\begin{proof}
	First note that, for each $n \in \mathbb{Z}_{>0}$ and each $\bm{x} \in P_n (X)$, the set $\Omega_{\bm{x}}$ is non-empty since $(X,\mathsf{d})$ is assumed to be geodesic. Thus, the $n$-th summand on the left hand side is
	\begin{align*}
		&\int_{\bm{x} \in P_n (X)}  \left(  \int_{\gamma \in \Omega_{\bm{x}}} e^{-\mathrm{Len} (\gamma) }  \mathrm{d} \Gamma^{\mathrm{triv}}_{\bm{x}} \right) \mathrm{d} \mu_{Z,\sharp}^{n+1} \\
		&= \int_{\bm{x} \in P_n(Z) }  e^{-  \sum_{k=1}^n \mathsf{d} (x_{k-1}, x_k)} \left( \prod_{k=0}^n \int_{\Omega_{x_{k-1}, x_k}} \mathrm{d}\Gamma^{\mathrm{triv}}_{x_{k-1}, x_k} \right) \mathrm{d} \mu_{Z, \sharp}^{n+1} \\
		&= \int_{\bm{x} \in P_n (Z)}  e^{-  \sum_{k=1}^n \mathsf{d} (x_{k-1}, x_k)} \mathrm{d} \mu_{Z,\sharp}^{n+1},
	\end{align*}
	irrespectively of whether each $\Omega_{x_{k-1}, x_k}$ is a singleton or not, since the measure $\Gamma^{\mathrm{triv}}$ is trivial and $\Omega_{x_{k-1}, x_k} \neq \emptyset$. The corresponding summand on the right hand side is
	\begin{equation*}
		\int_{\bm{x} \in P_n (Z)}  \left(  \int_{\gamma \in \Omega'_{\bm{x}}} e^{-\mathrm{Len} (\gamma) }  \mathrm{d} \Gamma^{\mathrm{triv}}_{\bm{x}} \right) \mathrm{d} \mu_{Z,\sharp}^{n+1} = \int_{\bm{x} \in P_n (Z)}  e^{- \sum_{k=1}^n \mathsf{d} (x_{k-1}, x_k)} \mathrm{d} \mu_{Z,\sharp}^{n+1},
	\end{equation*}
	where $\Omega'_{\bm{x}}$ is a singleton set corresponding to the proper chain $\bm{x} \in P_n (Z)$, by the convention explained at the beginning of this subsection. This gives the claimed equality for any $N \in \mathbb{Z}_{>0}$, and the convergence follows from Theorem \ref{lmcvfmsm}.
\end{proof}

\subsection{Finite graphs} \label{scitfms}

While the argument as above holds when we decree that any two points in a finite metric space are connected by a unique length-minimising geodesic, there are situations in which geodesics are defined more naturally. In such cases, it can happen that two points in a finite metric space are connected by non-unique geodesics. Concretely, such cases happen when we consider an undirected graph $G$ with no multiple edges, from which we can construct a finite metric space $(G, \mathsf{d}_{G})$ as follows: a path in $G$ is a collection of vertices, each connected by an edge, the length of a path is the number of edges connecting the vertices, and the metric $\mathsf{d}_G$ is the associated length metric. In this case we can naturally define length-minimising geodesics in $(G, \mathsf{d}_{G})$, and two points in this space may be connected by more than two geodesics as in the following example called 4-cut, shown in  Figure \ref{fig4cut} (see \cite[Definition 4.14]{LeiShu21} or \cite[Definition 2.2]{KY21}). The 4-cut has important consequences for the magnitude homology of graphs; see \cite[Corollary 4.6]{Asao}, \cite{Gomi}, and \cite[Theorem 4.21]{LeiShu21}. It is the simplest graph in which two points are connected by non-unique geodesics.

\begin{figure}[h]
\centering
\begin{tikzpicture}
\coordinate (A) at (0, 0);
\coordinate (B) at (2, 0);
\coordinate (C) at (2, 2);
\coordinate (D) at (0, 2);

\draw[thick] (A) -- (B) -- (C) -- (D) -- cycle;


\fill (A) circle (2pt) node[below left] {$x_0$};
\fill (B) circle (2pt) node[below right] {$x_1$};
\fill (C) circle (2pt) node[above right] {$x_2$};
\fill (D) circle (2pt) node[above left] {$x_3$};
\end{tikzpicture}
\caption{$4$-cut}
\label{fig4cut}
\end{figure}

In this case, non-trivial choices for $\Gamma$ lead to new quantities. When we take the counting measure $\Gamma^{\sharp}$, the $(\mu_{\sharp} , \Gamma^{\sharp})$-magnitude can be written by using an inverse matrix as follows, similarly to the original definition (\ref{eqtrdmag}).

\begin{lemma} \label{lmftgrmg}
Let $(G, \mathsf{d}_{G})$ be a finite undirected graph with the length metric as above. Let $\tilde{Z}_G$ be a $|G| \times |G|$ matrix whose $(x,y)$-th entry is defined by
\begin{equation*}
	(\tilde{Z}_G)_{x,y}:= e^{- \mathsf{d}_G (x,y)} \int_{\Omega_{x,y}} \mathrm{d} \Gamma^{\sharp}_{x,y} = e^{- \mathsf{d}_G (x,y)} |\Omega_{x,y}|
\end{equation*}
when $x \neq y$, and $(\tilde{Z}_G)_{x,y} :=1$ when $x=y$.

Then, the $(\mu_{\sharp} , \Gamma^{\sharp})$-magnitude of $(G, \mathsf{d}_G)$ is given by the sum of all entries of the matrix $\tilde{Z}^{-1}_G$ if $\Vert \tilde{Y}_G \Vert_{\mathrm{op}} <1$, where $\tilde{Y}_G := \tilde{Z}_G - I_{|G|}$ is defined similarly to \textup{(\ref{eqdfyxxy})}.
\end{lemma}

\begin{proof}
If $\Vert \tilde{Y}_G \Vert_{\mathrm{op}} <1$, $\tilde{Z}_{G}$ is invertible and the expansion $\tilde{Z}^{-1}_{G} = \sum_{n=0}^{\infty} (-1)^k \tilde{Y}^k_X$ converges absolutely. We thus repeat the computation in (\ref{eqzxivexp}) to get
\begin{align*}
	 &\sum_{x,y \in G} ((\tilde{Z}_G)^{-1})_{xy} \\
	 &= \mu_{\sharp} (G) +  \sum_{n=1}^{\infty} (-1)^n \int_{x \in G} \int_{y \in G} \sum_{\substack{x=x_0 \neq \cdots \neq x_{n}=y \\ x_2 , \dots , x_{n} \in X}} \prod_{k=1}^{n} |\Omega_{x_{k-1},x_k}| e^{- \mathsf{d}_G (x_{k-1}, x_k)} \mathrm{d} \mu_{\sharp} \mathrm{d} \mu_{\sharp} \\
	 &=\mu_{\sharp} (G) +  \sum_{n=1}^{\infty} (-1)^n \int_{P_n (G)} \left( \int_{\gamma \in \Omega_{\bm{x}}} e^{-\mathrm{Len} (\gamma)}  \mathrm{d} \Gamma^{\sharp}_{\bm{x}} \right) \mathrm{d} \mu^{n+1}_{\sharp} \\
	 &=\mathrm{Mag} (G, \mathsf{d}_G,\mu_{\sharp}, \Gamma^{\sharp}) 
\end{align*}
as required, recalling $\Omega_{x,y} = \emptyset$ when $x=y$.
\end{proof}

The condition $\Vert \tilde{Y}_G \Vert_{\mathrm{op}} <1$ can be satisfied by re-scaling $\mathsf{d}_G$ to $t \mathsf{d}_G$, where $t>0$ is chosen to be large enough so that
\begin{equation} \label{eqprfrsrdt}
	\max_{y \in G} \sum_{x \in G} |\Omega_{x,y}| e^{-t \mathsf{d}_G(x,y)} <1,
\end{equation}
similarly to (\ref{eqprfrsrd}). To simplify the notation, we define
\begin{equation*}
\widetilde{\mathrm{Mag}} (G, \mathsf{d}_{G}):= \sum_{x,y \in G} (\tilde{Z}^{-1}_G)_{xy}	.
\end{equation*}

\begin{example} \label{ex4cut}
When $G$ is a 4-cut, straightforward computation yields the classical magnitude function
\begin{equation*}
	\mathrm{Mag} (G, t \mathsf{d}_G) = \frac{4e^{2t}}{(1+e^{t})^2},
\end{equation*}
which agrees with the $(\mu_{\sharp} , \Gamma^{\mathrm{triv}})$-magnitude of $(G, t \mathsf{d}_G)$ by Theorem \ref{lmcvfmsm} when $t > \log 3 \approx 1.1$ by recalling (\ref{eqprfrsrd}). Similarly, we compute
\begin{equation*}
	\widetilde{\mathrm{Mag}} (G, t \mathsf{d}_{G}) = \frac{4e^{2t}}{4+e^{4t}} \left( 1+(1-e^t)^2 \right) ,
\end{equation*}
which agrees with the $(\mu_{\sharp} , \Gamma^{\sharp})$-magnitude of $(G, t \mathsf{d}_G)$ by Lemma \ref{lmftgrmg} when $t > \log 4 \approx 1.4$ by (\ref{eqprfrsrdt}). These two functions can be compared as in Figure \ref{fig4cstcps}.

\begin{figure}
\begin{tikzpicture}
\begin{axis}[
    axis lines = left,
    xlabel = {$t$},
    ylabel = {Magnitude},
    domain = 0:6, 
    samples = 200,
    grid = major,
    legend pos = south east,
    width = 10cm,
    height = 7cm,
    ymin = 0, ymax = 4.5,
    xmin = 0, xmax = 6,
    thick
]

\addplot [
    black,
] { (4 * exp(2*x)) / (1 + exp(x))^2 };
\addlegendentry{$\mathrm{Mag} (G, t \mathsf{d}_G)$}

\addplot [
    black,
    dashed,
] { (4 * exp(2*x) * (1 + (1 - exp(x))^2)) / (4 + exp(4*x)) };
\addlegendentry{$\widetilde{\mathrm{Mag}} (G, t \mathsf{d}_{G})$}

\end{axis}
\end{tikzpicture}
\caption{Comparison of the magnitude functions for $(G, t \mathsf{d}_G)$ in Example \ref{ex4cut}}
\label{fig4cstcps}
\end{figure}
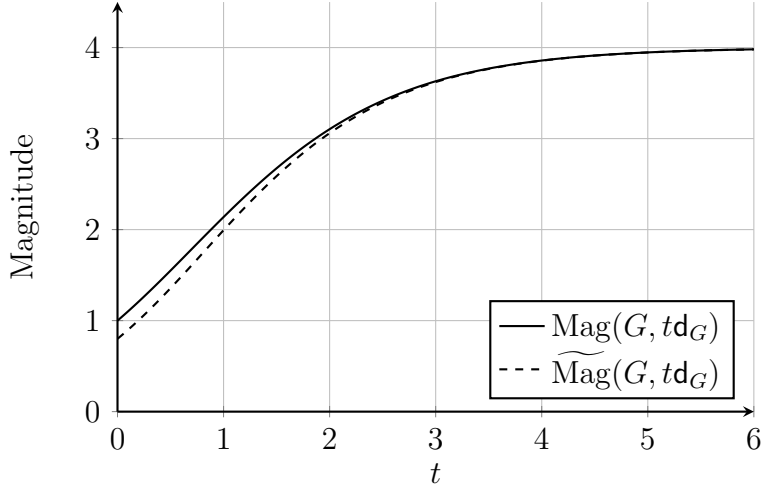
\end{example}

\begin{example} \label{ex4cutwgt}
If we allow ourselves to consider weighted graphs (although it is not commonly done in the literature on magnitude), we find that the $(\mu_{\sharp} , \Gamma^{\sharp})$-magnitude can distinguish the 4-cut above and the graph $G'$ in Figure \ref{fig4cutedg} which has an extra diagonal edge of length 2. In this case, we can regard $x_0$ and $x_2$ to be connected by three geodesics.

This extra diagonal does not affect the similarity matrix, since $(G, \mathsf{d}_G)$ and $(G', \mathsf{d}_{G'})$ define the same metric space. In particular, the magnitude of $(G', \mathsf{d}_{G'})$ is the same as that of $(G, \mathsf{d}_G)$. We thus get
\begin{equation*}
	 \mathrm{Mag} (G', t \mathsf{d}_{G'}) = \frac{4e^{2t}}{(1+e^{t})^2}
\end{equation*}
from Example \ref{ex4cut}, agreeing with $\mathrm{Mag} (G', t \mathsf{d}_{G'},\mu_{\sharp}, \Gamma^{\mathrm{triv}})$ for $t > \log 3 \approx 1.1$, but direct computation using Lemma \ref{lmftgrmg} yields
\begin{equation*}
	 \widetilde{\mathrm{Mag}} (G', t \mathsf{d}_{G'})  = \frac{4e^{2t}}{6+e^{2t}+e^{4t}} (2-2e^t+e^{2t}), 
\end{equation*}
which agrees with $\mathrm{Mag} (G', t \mathsf{d}_{G'},\mu_{\sharp}, \Gamma^{\sharp})$ for $t > \log 5 \approx 1.6$, but it does not equal $\mathrm{Mag} (G', t \mathsf{d}_{G'})$ (see Figure \ref{fig4cstcpsp}). What this result shows is that the $(\mu_{\sharp} , \Gamma^{\sharp})$-magnitude is not an invariant of finite metric spaces when it is extended to weighted graphs, and that it encodes finer information about the weighted graph, in terms of the number of geodesics (edges) connecting distinct vertices.

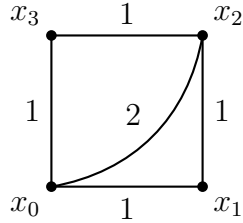
\begin{figure}[h]
\centering
\begin{tikzpicture}
\coordinate (A) at (0, 0);
\coordinate (B) at (2, 0);
\coordinate (C) at (2, 2);
\coordinate (D) at (0, 2);

\draw[thick] (A) -- node[auto=right]{$1$} (B) -- node[auto=right]{$1$} (C) -- node[auto=right]{$1$} (D) -- node[auto=right]{$1$}cycle;

\draw[thick] (A) to [out=10,in=-100] node[auto=left]{$2$} (C);

\fill (A) circle (2pt) node[below left] {$x_0$};
\fill (B) circle (2pt) node[below right] {$x_1$};
\fill (C) circle (2pt) node[above right] {$x_2$};
\fill (D) circle (2pt) node[above left] {$x_3$};
\end{tikzpicture}
\caption{$4$-cut with an extra edge}
\label{fig4cutedg}
\end{figure}

\begin{figure}
\begin{tikzpicture}
\begin{axis}[
    axis lines = left,
    xlabel = {$t$},
    ylabel = {Magnitude},
    domain = 0:6, 
    samples = 200,
    grid = major,
    legend pos = south east,
    width = 10cm,
    height = 7cm,
    ymin = 0, ymax = 4.5,
    xmin = 0, xmax = 6,
    thick
]

\addplot [
    black,
] { (4 * exp(2*x)) / (1 + exp(x))^2 };
\addlegendentry{$\mathrm{Mag} (G', t \mathsf{d}_{G'})$}

\addplot [
    black,
    dashed,
] { (4 * exp(2*x) * (2 - 2 * exp(x) + exp(2*x)) / (6 + exp(2*x) + exp(4*x)) };
\addlegendentry{$\widetilde{\mathrm{Mag}} (G', t \mathsf{d}_{G'})$}

\end{axis}
\end{tikzpicture}
\caption{Comparison of the magnitude functions for $(G', t \mathsf{d}_{G'})$ in Example \ref{ex4cutwgt}}
\label{fig4cstcpsp}
\end{figure}
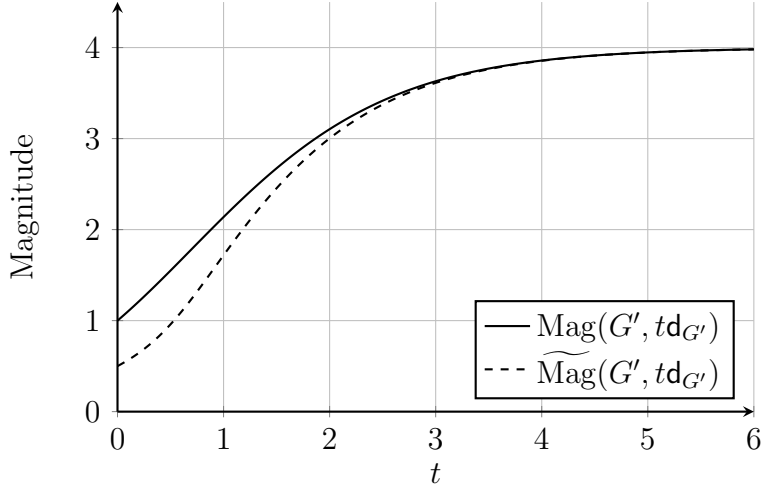

\end{example}

\subsection{Convergence for compact geodesic metric spaces} \label{sccfcgms}

The definition of the $(\mu, \Gamma)$-magnitude is given in terms of an infinite series, which may not converge in general. We give two sufficient conditions for the convergence when $(X, \mathsf{d})$ is a compact geodesic metric space, and prove some continuity results for the $(\mu , \Gamma^{\mathrm{triv}})$-magnitude with respect to the measure $\mu$.

We first define the following condition which is satisfied for a wide range of examples, e.g.~when $(X, \mathsf{d})$ is a Riemannian manifold and $\mu$ is absolutely continuous with respect to the Lebesgue measure.

\begin{definition}
	We say that a Borel measure $\mu$ on a metric space $(X , \mathsf{d})$ \textbf{puts no mass on non-proper chains} if the product measure $\mu^{n+1}$ satisfies $\mu^{n+1} (X^{n+1} \setminus P_n (X)) =0$ for any $n \in \mathbb{Z}_{>0}$.
\end{definition}

Henceforth we assume that $(X, \mathsf{d})$ is geodesic metric space, and that $\mu$ is a Radon measure. An important simplification happens in this case since $\Omega_{\bm{x}} \neq \emptyset$ for any $\bm{x} \in P_n (X)$, which implies
\begin{equation*}
	\Gamma^{\mathrm{triv}}_{\bm{x}} (\Omega_{\bm{x}}) =1 
\end{equation*}
for any $\bm{x} \in P_n (X)$. Together with the definition above, we immediately have the following.

\begin{theorem} \label{thgmsnmnpc}
	Suppose that $(X, \mathsf{d})$ is geodesic metric space and $\mu$ is a Radon measure which puts no mass on non-proper chains. Then for any $N \in \mathbb{Z}_{>0}$ we have
	\begin{equation*}
		\mathrm{Mag}(X, \mathsf{d}, \mu , \Gamma^{\mathrm{triv}};N) =\mu (X) +\sum_{n=1}^{N} (-1)^n \int_{X^{n+1}} e^{-\sum_{k=1}^n \mathsf{d}(x_{k}, x_{k+1})} \mathrm{d} \mu^{n+1}.
	\end{equation*}
	In particular, $\mathrm{Mag}(X, \mathsf{d}, \mu , \Gamma^{\mathrm{triv}};N) = \mathrm{Mag}(X, \mathsf{d}, \mu ;N)$.
\end{theorem}

The convergence as $N \to \infty$ can be guaranteed by the following theorem.

\begin{theorem} \label{ppexmgpm}
	Suppose that $(X,\mathsf{d})$ is a compact geodesic metric space and $\mu$ is a Radon probability measure which puts no mass on non-proper chains, and that the support of $\mu$ is not a singleton set. Then
	\begin{equation*}
		\mathrm{Mag} (X,\mathsf{d}, \mu , \Gamma^{\mathrm{triv}})  = \lim_{N \to \infty} \mathrm{Mag}(X, \mathsf{d}, \mu , \Gamma^{\mathrm{triv}};N)
	\end{equation*}
	exists as a real number.
\end{theorem}

\begin{proof}
	By the Lebesgue convergence theorem, we find that the map $X \to \mathbb{R}$ defined by
	\begin{equation*}
		y \mapsto \int_{x \in X} e^{- \mathsf{d} (x,y) } \mathrm{d} \mu 
	\end{equation*}
	is continuous. Since $X$ is compact, there exists $z \in X$ such that
	\begin{equation*}
		\int_{x \in X} e^{- \mathsf{d} (x,z) } \mathrm{d} \mu = \sup_{y \in X} \int_{x \in X} e^{- \mathsf{d} (x,y) } \mathrm{d} \mu .
	\end{equation*}
	We now find $c:= \int_X e^{- \mathsf{d} (x,z) } \mathrm{d} \mu (x) < 1$, since otherwise $e^{- \mathsf{d} (x,z) } \equiv 1$ for any $x \in \mathrm{supp} (\mu )$, which is a contradiction (note that the support of $\mu$ is not a singleton set). We then find, for any $n \ge 1$, that
	\begin{equation*}
		\int_{X^{n+1}}  e^{-\sum_{k=1}^n \mathsf{d} (x_{k-1}, x_k)} \mathrm{d} \mu^{n+1} \le c^{n}
	\end{equation*}
	which ensures the convergence of the infinite sum.
\end{proof}

The theorem above also works for a measure with $\mu (X) \le 1$. When $\mu(X) >1$, there are two options to get the convergence: one is to re-scale $\mu$ to make it a probability measure, and the other is to re-scale the metric $\mathsf{d}$. Re-scaling the measure turns out to change the magnitude in a highly non-trivial manner (see Proposition \ref{ppscmdc} or Remark \ref{rmscmtvf}). We discuss here the following convergence result when we re-scale the metric.

\begin{theorem} \label{ppexmgpmlt}
	Suppose that $(X,\mathsf{d})$ is a compact geodesic metric space and $\mu$ is a Radon measure which puts no mass on non-proper chains, and that for any $x \in X$ there exists a neighbourhood $N_x$ such that $\mu (N_x) <1$. Then, there exists $t_0 >0$ such that $\mathrm{Mag} (X, t \mathsf{d}, \mu , \Gamma^{\mathrm{triv}})$ exists as a real number for all $t \ge t_0$.
\end{theorem}

The assumption for the measure is satisfied for a wide range of examples, e.g.~when $(X, \mathsf{d})$ is a complete Riemannian manifold and $\mu$ is absolutely continuous with respect to the Lebesgue measure.

\begin{proof}
	It suffices to show, as before, that there exists $t_0 > 0$ such that
	\begin{equation*}
		\max_{y \in X} \int_{x \in X} e^{- t \mathsf{d} (x,y) } \mathrm{d} \mu  < 1 
	\end{equation*}
	holds for all $t \ge t_0$, where we note that the maximum exists by compactness and continuity as before. Assuming otherwise for contradiction, we get a sequence $\{ y_j \}_{j=1}^{\infty}$ in $X$ such that 
	\begin{equation*}
		\int_{x \in X} e^{- j \mathsf{d} (x,y_j) } \mathrm{d} \mu  \ge 1
	\end{equation*}
	holds for all $j \in \mathbb{Z}_{>0}$. Since $(X, \mathsf{d})$ is compact, we may pass to a convergent subsequence, still denoted by $\{ y_j \}_{j=1}^{\infty}$, which converges to $y_{\infty} \in X$. We then pick a neighbourhood $N_{\infty}$ of $y_{\infty}$, which we may assume contains an open ball $B_r (y_{\infty})$ of radius $r>0$, such that $\mu (N_{\infty}) <1$. We then get
	\begin{equation} \label{eqpflbjdv}
		\int_{x \in N_{\infty}} e^{- j \mathsf{d} (x,y_j) } \mathrm{d} \mu  + \int_{x \in X \setminus N_{\infty}} e^{- j ( \mathsf{d} (x,y_{\infty}) - \mathsf{d}(y_j , y_{\infty}))} \mathrm{d} \mu  \ge 1
	\end{equation}
	from the triangle inequality. The first term can be bounded by
	\begin{equation*}
		\int_{x \in N_{\infty}} e^{- j \mathsf{d} (x,y_j) } \mathrm{d} \mu  \le \mu (N_{\infty}) <1
	\end{equation*}
	uniformly for any $j$, which follows from our hypothesis on $\mu$. To evaluate the second term, we pick $j_0 \in \mathbb{Z}_{>0}$ so that $\mathsf{d}(y_j , y_{\infty}) < r/2$ holds for all $j \ge j_{0}$. We then have
	\begin{equation*}
		\int_{x \in X \setminus N_{\infty}} e^{- j ( \mathsf{d} (x,y_{\infty}) - \mathsf{d}(y_j , y_{\infty}))} \mathrm{d} \mu  < e^{-jr/2} \mu(X) \to 0
	\end{equation*}
	as $j \to \infty$, which contradicts (\ref{eqpflbjdv}).
\end{proof}

The magnitude is continuous with respect to a weakly convergent sequence of measures under some assumptions, as below.

\begin{proposition} \label{thwcvpmmj}
	Let $\{ \mu_j \}_{j=1}^{\infty}$ be a sequence of Radon probability measures on a compact geodesic metric space $(X,\mathsf{d})$, converging weakly to a Radon probability measure $\mu_{\infty}$. Suppose that all these measures put no mass on non-proper chains and that the support is not a singleton set for any of $\{ \mu_j \}_{j=1}^{\infty} \cup \{ \mu_{\infty} \}$. Then we have
	\begin{equation*}
		\mathrm{Mag} (X, \mathsf{d}, \mu_{\infty} , \Gamma^{\mathrm{triv}}) = \lim_{j \to \infty} \mathrm{Mag} (X, \mathsf{d}, \mu_j , \Gamma^{\mathrm{triv}}).
	\end{equation*}
\end{proposition}

\begin{proof}
	It suffices to show that there exists $c \in (0,1)$ such that
	\begin{equation*}
		\max_{y \in X} \int_{x \in X} e^{-  \mathsf{d} (x,y) } \mathrm{d} \mu_k  \le c <1
	\end{equation*}
	holds uniformly for all $k \in \mathbb{Z}_{>0} \cup \{ \infty \}$, but this follows from the weak convergence of measures and the assumptions of the proposition.
\end{proof}

\begin{proposition} \label{thwcvmj}
	Let $\{ \mu_j \}_{j=1}^{\infty}$ be a sequence of Radon measures on a compact geodesic metric space $(X,\mathsf{d})$, converging weakly to a Radon measure $\mu_{\infty}$. Suppose that all these measures put no mass on non-proper chains, and that for each $x \in X$ there exists a neighbourhood $N_{x}$ such that $\mu_k (N_{x}) < 1$ holds uniformly for all $k \in \mathbb{Z}_{>0} \cup \{ \infty \}$. Then there exists $t_0 >0$ such that 
	\begin{equation*}
		\mathrm{Mag} (X, t \mathsf{d}, \mu_{\infty} , \Gamma^{\mathrm{triv}}) = \lim_{j \to \infty} \mathrm{Mag} (X, t \mathsf{d}, \mu_j , \Gamma^{\mathrm{triv}})
	\end{equation*}
	holds for all $t \ge t_0$.
\end{proposition}

\begin{proof}
	Again it suffices to show the uniform convergence, which follows from the argument in the proof of Theorem \ref{ppexmgpmlt} and Proposition \ref{thwcvpmmj}.
\end{proof}

\subsection{Approximation by measures with finite support and Fekete configurations}

The classical magnitude (\ref{eqdfclmg}) for an infinite metric space $(X, \mathsf{d})$ is defined by taking the supremum over all finite subsets. By using the counting measure $\mu_{Z , \sharp}$ on a finite subset $Z \subset X$, this quantity can also be written as
	\begin{equation*}
		\sup_{Z \subset X , \; \text{finite}} \mathrm{Mag} (X,\mathsf{d},\mu_{Z, \sharp}, \Gamma^{\mathrm{triv}}),
	\end{equation*}
by Lemma \ref{lmfsbepm} and Theorem \ref{lmcvfmsm}, if the series converges for all $Z$.

This point of view provides a new perspective on the relationship between our definition of magnitude $\mathrm{Mag} (X,\mathsf{d},\mu, \Gamma^{\mathrm{triv}})$ and the classical magnitude for finite metric spaces, in terms of approximating the measure $\mu$ by ones supported on finitely many points. For example, approximating a probability measure by a sequence of measures of finite support is a well-studied problem called the quantisation of measures (see e.g.~\cite{GL00}), with a well-established algorithm for finding the finite points called Lloyd's algorithm \cite[Chapter 1, \S 4]{GL00}.

For such approximations, it is natural to consider what happens when we use the empirical measure, in place of the counting measure, which is defined as an average of delta measures supported on a finite subset $Z \subset X$ as
\begin{equation*}
	\mu_{Z, \mathrm{em}} := \frac{1}{|Z|} \sum_{z \in Z}  \delta_{z}.
\end{equation*}
A merit of working with the empirical measure is the following compactness result: if $(X, \mathsf{d})$ is separable, Prokhorov's theorem implies that there exists a subsequence in the family $\{ \mu_{Z, \mathrm{em}} \}_{Z \subset X, |Z| < \infty}$ converging weakly to a measure $\mu_{\infty}$ on $X$, which is a probability measure if the collection $\{ \mu_{Z, \mathrm{em}} \}_{Z \subset X, |Z| < \infty}$ is tight. Note that the scaling by $1/|Z|$ is crucial in this argument to have a uniform bound on the total volume. Thus, we find, up to taking a subsequence,
	\begin{equation} \label{eqcvgpkr}
	\mathrm{Mag} (X,\mathsf{d},\mu_{\infty} ;N) = \lim_{|Z| \to \infty} \mathrm{Mag} (X,\mathsf{d},\mu_{Z, \mathrm{em}} ;N)
	\end{equation}
	for all $N \in \mathbb{Z}_{>0}$, by the weak convergence. The convergence $N \to \infty$ is more subtle as it involves exchanging the limits, but possible e.g.~when the assumptions for Propositions \ref{thwcvpmmj} and \ref{thwcvmj} hold.

The empirical measure differs from the counting measure $\mu_{Z, \sharp}$ (in Lemma \ref{lmfsbepm}) only by an overall scaling. We start with the following observation, to see how the scaling affects the magnitude.

\begin{proposition} \label{ppscmdc}
	Let $c>0$ be a constant and let $(Z,\mathsf{d} , \mu_{\sharp})$ be a finite metric space with the counting measure. Let $\mathsf{d}_c$ be a metric on $Z$ defined by
	\begin{equation*}
		\mathsf{d}_c (x,y) := \begin{cases}
			0 &\quad (x=y) \\
			\mathsf{d} (x,y) +c &\quad (x \neq y).
		\end{cases}
	\end{equation*}
	Then for any $N \in \mathbb{Z}_{>0}$ and for any measure $\Gamma$ on the set of piecewise geodesics, we have
	\begin{equation*}
		\mathrm{Mag} (Z,\mathsf{d}, e^{-c} \mu_{\sharp}, \Gamma ; N) = e^{-c} \mathrm{Mag} (Z,\mathsf{d}_{c}, \mu_{\sharp}, \Gamma ; N).
	\end{equation*}
\end{proposition}

\begin{proof}
	The claim follows from
	\begin{align*}
		&\int_{\bm{x} \in P_n (X)}  \left(  \int_{\gamma \in \Omega_{\bm{x}}} e^{- \mathrm{Len} (\gamma) }  \mathrm{d} \Gamma_{\bm{x}} \right) \mathrm{d} (e^{-c}\mu_{\sharp})^{n+1} \\
		&= e^{-c} \int_{\bm{x} \in P_n (X)}  e^{- \sum_{k=1}^n (d(x_{k-1}, x_k)+c)} \left( \prod_{k=0}^n \int_{\gamma_k \in \Omega_{x_{k-1}, x_k}} \mathrm{d}\Gamma_{x_{k-1}, x_k} \right) \mathrm{d} \mu_{\sharp}^{n+1}
	\end{align*}
	for any $n=1 , \dots , N$.
\end{proof}

The computation above holds for general metric measure spaces, but we restricted to finite metric spaces since $\mathsf{d}_c$ induces the discrete topology on $Z$ irrespectively of the one on $(Z, \mathsf{d})$. Since the measure needs to be a Borel measure for both $\mathsf{d}$ and $\mathsf{d}_c$, it is natural to restrict to a discrete space.

\begin{example}
	By Proposition \ref{ppscmdc}, the limit in (\ref{eqcvgpkr}) can also be written as
	\begin{equation*}
		\mathrm{Mag} (X,\mathsf{d},\mu_{\infty} ;N) = \lim_{|Z| \to \infty} \frac{1}{|Z|} \mathrm{Mag} (X,\mathsf{d}_{\log |Z|},\mu_{Z , \sharp} ;N).
	\end{equation*}
\end{example}

Let $X$ be a compact complex manifold and $K$ be a regular non-pluripolar compact subset of $X$ (e.g.~smoothly bounded domain in $\mathbb{C}^n$, for the case when $X = \mathbb{CP}^n$, see \cite[Definition 3.3, Theorem 3.4]{BB10}), endowed with a continuous plurisubharmonic weight $\phi_{\mathrm{psh}}$ and a Riemannian volume form defined by a hermitian metric on $X$. Then, a theorem due to Berman--Boucksom--Witt Nystr\"om \cite[Theorem A]{BBW} states that the empirical measure associated a configuration of $m$ points in $K$ called a Fekete configuration converges to the equilibrium measure on $K$ (see \cite[\S 1.2]{BB10}). The Fekete configuration maximises an energy function which can be regarded as a generalisation of the Coulomb repulsion force, and hence they tend to be evenly spread (see e.g.~\cite[Remarque 2.8]{Duj20}). 

\begin{theorem} \label{thmfkeqm}
	Let $K$ be a compact non-pluripolar subset in a compact complex manifold $X$ with a plurisubharmonic weight $\phi_{\mathrm{psh}}$, and $\mathsf{d}$ be a distance on $K$ induced from a hermitian metric on $X$. Let $\mu_m$ be the empirical measure of the Fekete configuration $\mathcal{F}_m \subset K$ with respect to $\phi_{\mathrm{psh}}$, where $|\mathcal{F}_m|=m$. Then we have
	\begin{equation*}
		\mathrm{Mag} (K, \mathsf{d}, \mu_{\mathrm{eq}} , \Gamma^{\mathrm{triv}})  = \lim_{m \to \infty} \frac{1}{m}  \mathrm{Mag} ( \mathcal{F}_m , \mathsf{d}_{\log m} )
	\end{equation*}
	where $\mu_{\mathrm{eq}}$ is the equilibrium measure of $K$ and $\mathsf{d}_{\log m}$ is the metric on $\mathcal{F}_m$ as defined in Proposition \ref{ppscmdc}.
\end{theorem}

\begin{proof}
	By Berman--Boucksom--Witt Nystr\"om \cite[Theorem A]{BBW}, we easily have
	\begin{equation*}
		\mathrm{Mag} (K, \mathsf{d}, \mu_{\mathrm{eq}} , \Gamma^{\mathrm{triv}}; N) = \lim_{m \to \infty}  \mathrm{Mag} ( \mathcal{F}_m , \mathsf{d} |_{\mathcal{F}_m} , \mu_{\mathcal{F}_m, \mathrm{em}} , \Gamma^{\mathrm{triv}};N)
	\end{equation*}
	for any $N \in \mathbb{Z}_{>0}$, since $\mu_{\mathcal{F}_m, \mathrm{em}}$ converges weakly to $\mu_{\mathrm{eq}}$. Since $K$ is compact and $\mu_{\mathrm{eq}}$ is an absolutely continuous probability measure \cite[Corollary 2.12]{BB10}, the limit $N \to \infty$ on the left hand side exists by Theorem \ref{ppexmgpm}. It thus suffices to show that the convergence $N \to \infty$ is uniform for all $m$ on the right hand side, in order to exchange the limits $m \to \infty$ and $N \to \infty$. As in the proof of Proposition \ref{thwcvmj}, it suffices to show that there exists $c \in (0,1)$ such that
	\begin{equation} \label{eqpmempf}
		\max_{y \in K} \int_{x \in K} e^{-  \mathsf{d} (x,y) } \mathrm{d} \mu_{\mathcal{F}_m, \mathrm{em}}  \le c <1
	\end{equation}
	holds uniformly for all $m \in \mathbb{Z}_{>0}$, but this follows from the fact that $\mu_{\mathcal{F}_m, \mathrm{em}}$ converges weakly to $\mu_{\mathrm{eq}}$ and that $\mu_{\mathrm{eq}}$ is an absolutely continuous probability measure, by recalling the argument used in the proof of Theorem \ref{ppexmgpm}. We finally get the result by noting
	\begin{equation*}
		\mathrm{Mag} ( \mathcal{F}_m , \mathsf{d} |_{\mathcal{F}_m} , \mu_{\mathcal{F}_m, \mathrm{em}} , \Gamma^{\mathrm{triv}}) = \frac{1}{m}  \mathrm{Mag} ( \mathcal{F}_m , \mathsf{d}_{\log m} ),
	\end{equation*}
	which follows from Proposition \ref{ppscmdc}, Lemma \ref{lmfsbepm}, and Theorem \ref{lmcvfmsm}, by noting that the convergence follows from (\ref{eqpmempf}) above.
\end{proof}

With the intuition of the Fekete configuration being evenly spread, it is tempting to speculate that the classical definition of magnitude (\ref{eqdfclmg}) is related to the magnitude of the Fekete configuration. While we do have a formula relating these two notions as in Corollary \ref{thmclmfkc1}, we do not pursue any general results in this paper.

\section{Riemannian manifolds} \label{scrm}

Let $(X,g)$ be an oriented Riemannian manifold. Throughout this section, the distance $\mathsf{d} = \mathsf{d}_g$ is meant to be the intrinsic metric induced from the Riemannian metric $g$. Together with a Radon measure $\mu$, it naturally defines a metric measure space $(X , \mathsf{d} , \mu )$.

\subsection{Preliminary results}

We start with the following lemma.

\begin{lemma} \label{lmacompn}
	Suppose that $\mu$ is absolutely continuous with respect to the Lebesgue measure. Then, for any $n \in \mathbb{Z}_{>0}$, $\mu$ puts no mass on non-proper chains $X^{n+1} \setminus P_n (X)$, and $\Omega_{\bm{x}}$ is a singleton set for any $\bm{x} \in P_n (X)$ outside of a set of measure zero with respect to $\mu^{n+1}$.
\end{lemma}

\begin{proof}
	The first claim is obvious since the set of non-proper chains is a proper submanifold of $X^{n+1}$. The second claim is a consequence of the well-known fact that a cut locus has measure zero (see e.g.~\cite[Lemma 3.96]{GHL04}).
\end{proof}

This lemma implies that the choice of the measure $\Gamma_{\bm{x}}$ on $\Omega_{\bm{x}}$ is irrelevant for the definition of the $(\mu , \Gamma )$-magnitude, as long as $\Gamma_{\bm{x}} ( \Omega_{\bm{x}} ) =1$ when $\Omega_{\bm{x}}$ is a singleton set.

\begin{remark}
	It is also known that, for an $\textsf{RCD}^* (K,N)$ space with $K \in \mathbb{R}$ and $N \in [1, \infty )$, any point can be connected to another by a unique geodesic almost everywhere \cite[Corollary 1.4]{GRS16}.
\end{remark}

When the Riemannian manifold $(X,g)$ is complete, the metric space $(X , \mathsf{d})$ is a geodesic metric space by the Hopf--Rinow theorem. Furthermore, the assumption of Lemma \ref{lmacompn} is satisfied when the measure $\mu$ is induced from the weighted volume form, i.e.~when there exists a smooth function $\psi$ on $X$ such that
\begin{equation*}
	\mathrm{d} \mu = e^{\psi} \mathrm{d} \mathrm{vol}_g
\end{equation*}
where $\mathrm{d} \mathrm{vol}_g$ is the volume form of $g$. Thus a simplified formula for the magnitude in Theorem \ref{thgmsnmnpc} is available in this case, so that
\begin{align*}
	\mathrm{Mag}(X, \mathsf{d}, \mu , \Gamma ;N) &= \mathrm{Mag}(X, \mathsf{d}, \mu ;N) \\
	&=\mu (X) +\sum_{n=1}^{N} (-1)^n \int_{X^{n+1}} e^{-\sum_{k=1}^n \mathsf{d}(x_{k}, x_{k+1})} \mathrm{d} \mu^{n+1}
\end{align*}
for any measure $\Gamma$ which satisfies $\Gamma_{\bm{x}} ( \Omega_{\bm{x}} ) =1$ when $\Omega_{\bm{x}}$ is a singleton set (which includes $\Gamma^{\mathrm{triv}}$).

We thus focus on the $\mu$-magnitude throughout this section. Hence computing the magnitude reduces to computing the series of integrals
\begin{equation*}
	\int_{X^{n+1}} e^{-\sum_{k=1}^n \mathsf{d}(x_{k}, x_{k+1})} \mathrm{d} \mu^{n+1}
\end{equation*}
for all $n \in \mathbb{Z}_{>0}$. In this section, we do not discuss the limit $N \to \infty$ since the main point is the computation of the above integral, for each fixed $n$. The convergence is guaranteed for the case when the results in \S \ref{sccfcgms} are available.

\begin{lemma} \label{lmdcvfdml}
Let $\mu$ be the weighted volume form on a complete Riemannian manifold $(X,g)$. The following subset
\begin{equation*}
	P_n^l (X) := \left\{ (x_0, \dots , x_n ) \in P_n (X) \; \left| \; \sum_{k=1}^n \mathsf{d} (x_{k-1}, x_k)= l \right\} \right. 
\end{equation*}
is a smooth hypersurface of $P_n (X)$ outside a set of measure zero with respect to $\mu$. Moreover there exists a volume form $\mu_l^n$ on $P_n^l (X)$ outside a set of measure zero with respect to $\mu$, such that
\begin{equation*}
	\mathrm{d} l \cdot \mathrm{d} \mu_l^n =  \mathrm{d} \mu^{n+1}.
\end{equation*}
\end{lemma}

\begin{proof}
	The function $\sum_{k=1}^n \mathsf{d} (x_{k-1}, x_k)$ is smooth on $P_n (X)$ except when $x_k$ is in the cut locus of $x_{k-1}$ for some $k=1, \dots , n$, which is of measure zero. Thus, outside this set of measure zero, $\sum_{k=1}^n \mathsf{d} (x_{k-1}, x_k)$ is smooth on $P_n (X)$ with clearly non-vanishing derivative, and hence the claim follows from the implicit function theorem. The second claim is obvious by taking $l$ as the coordinate function.
\end{proof}

The above lemma and change of variables immediately imply the following result, which can be regarded as a generalisation of (\ref{eqcmpilsh}).

\begin{theorem} \label{thmggrmmf}
	Suppose that $(X,g)$ is a complete Riemannian manifold and $\mu$ is defined by a smooth weighted volume form such that $\mu (X)$ is finite. For each $n \in \mathbb{Z}_{>0}$, we have
	\begin{equation*}
	\int_{X^{n+1}} e^{-\sum_{k=1}^n \mathsf{d}(x_{k}, x_{k+1})} \mathrm{d} \mu^{n+1} = \int_0^{\infty} e^{-l } \mu_l^n (P_n^l(X) ) \mathrm{d} l ,
	\end{equation*}
	where $\mu_l^n (P_n^l(X) ) := \int_{P^l_n (X)} \mathrm{d} \mu_l^n$ by using the volume form in Lemma \ref{lmdcvfdml}. In particular, the $N$-th partial $\mu$-magnitude can be written as
	\begin{equation*}
		\mathrm{Mag} (X,\mathsf{d},\mu; N) = \mu (X) + \sum_{n=1}^{N} (-1)^n \int_0^{\infty} e^{-l } \mu_l^n (P_{n}^{l} (X) ) \mathrm{d} l 
	\end{equation*}
	for all $N \in \mathbb{Z}_{>0}$.
\end{theorem}

\subsection{Examples from uniquely geodesic manifolds} \label{scexmfduqg}

As we shall see in the examples later, the cut locus and non-uniqueness of geodesics play an important role in the expansion of the magnitude function. When $(X,g)$ is a complete Riemannian manifold that is uniquely geodesic, which in particular has an empty cut locus, the magnitude function behaves very differently to the one with a non-trivial cut locus, as we see in the examples below. In this subsection $(X,g)$ is assumed to be a uniquely geodesic complete Riemannian manifold, but we may focus on Hadamard manifolds without losing the essence of the arguments.
\begin{lemma} \label{lmuqgdvlsm}
	Let $(X,g)$ is a complete Riemannian manifold which is uniquely geodesic. Let $\mu$ be measure of finite volume on $X$ induced by a smooth weighted volume form, $n \in \mathbb{Z}_{>0}$, and let $\mu_l^n$ be the induced volume form on $P^l_n (X)$. Then $\mu_l^n (P^l_n (X))$ is differentiable infinitely many times in $l$ for all $l > 0$.
\end{lemma}

\begin{proof}
	Since $(X,g)$ is uniquely geodesic, the cut locus is empty for any point $p \in X$. We then find that the sum of the distance functions
	\begin{equation*}
		\mathsf{d} (x_0, x_1) + \cdots + \mathsf{d} (x_{n-1}, x_n) =: l (x_0 , \dots , x_n)
	\end{equation*}
	is smooth everywhere on $P_n (X)$. As $l : P_n (X) \to \rl$ is clearly a submersion, we immediately find that $P^l_n (X)$ is a smooth hypersurface in $P_n (X)$ for any $l > 0$. Moreover, when we fix the length $\mathsf{d} (x_{j-1}, x_j)$ to be $s_j >0$ for each $j=1 , \dots , n$, we find that $P_n^l (X)$ can be decomposed into nested geodesic spheres in $X$ as
	\begin{equation*}
		P_n^l (X) = \bigsqcup_{s_1+ \cdots + s_n=l} \bigsqcup_{x_0 \in X}  \{ (x_0 , \dots , x_n ) \in P_n(X) \mid x_j \in S (x_{j-1}; s_j) , \; j=1 , \dots , n \} ,
	\end{equation*}
	where $S (x_{j-1}; s_j)$ is the geodesic sphere of radius $s_j$ around $x_{j-1}$.
	
	Suppose that we write $\mathrm{d} \mu_{s_j}$ for the volume form on the geodesic sphere $S (x_{j-1}; s_j)$ induced from the smooth volume form $\mathrm{d} \mu$ on $X$, noting that $s_j$ defines a smooth coordinate on $X$ as above. Then, writing
	\begin{equation*}
	Q_n^l := \{ s_1 , \dots , s_n >0 , \; s_1+ \cdots + s_n=l \} \subset \mathbb{R}^n ,
	\end{equation*}
	and by using the $(n-1)$-dimensional Hausdorff measure $\mathrm{d} s$ on $Q_{n,l}$ induced by the Euclidean metric, we find
	\begin{equation*}
		\mu_l^n (P^l_n (X)) = \int_{Q_n^l} \left( \int_{x_0 \in X} \int_{x_{1} \in S (x_{0}; s_1)} \cdots \int_{x_{n} \in S (x_{n-1}; s_n)} \mathrm{d} \mu_{s_n} \cdots \mathrm{d} \mu_{s_1} \mathrm{d} \mu \right) \mathrm{d} s ,
	\end{equation*}
	which is finite since geodesic spheres are compact. Since the volume of a geodesic sphere depends smoothly on its radius, we find that $\mu_l^n (P^l_n (X))$ depends smoothly on $l$.
\end{proof}

The crucial hypothesis of unique geodesics is used not only to ensure that the distance functions are smooth, but also in the range of integration $Q_n^l$ in the proof above: the only constraint amongst $s_1 , \dots , s_n >0$ is $s_1+ \cdots + s_n=l$, which obviously varies smoothly as $s_1 , \dots , s_n$ vary. However, when $X$ has a cut locus, the range of $s_1 , \dots , s_n$ has further constraints as we see in \S \ref{scexmfdcl}, which results in $\mu_l^n (P^l_n (X))$ being non-differentiable; compare Figures \ref{figsqhdmnc} and \ref{figsqcptcd}.

\begin{figure}[h!]
\centering
\begin{tikzpicture}[scale=2.5] 

    \fill[gray!10] (0, 0) -- (1.2, 0) -- (1.2, 1.2) -- (0, 1.2) -- cycle;  
    \draw[->, thin] (-0.1, 0) -- (1.2, 0) node[right] {$s_1$};
    \draw[->, thin] (0, -0.1) -- (0, 1.2) node[above] {$s_2$};
	\node at (1, -0.1)  {$l$};
	\node at (-0.1, 1) {$l$};
        
    \draw[line width=1.5pt, ] (0, 1) -- (1, 0);

\end{tikzpicture}
\caption{The length of the solid line varies smoothly.}
\label{figsqhdmnc}
\end{figure}
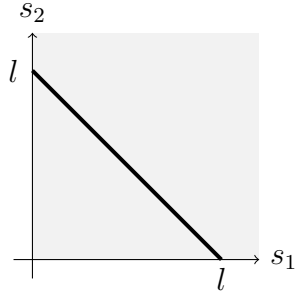

In any case, when $(X,g)$ is uniquely geodesic, $\mu_l^n (P^l_n (X))$ is smooth. Consider now a hypothetical situation when $\mu_l^n (P^l_n (X))$ is a polynomial in $l$. In this case, writing $\mu_l^n (P^l_n (X)) = \sum_{m=1}^{M} a_{m} \frac{l^m}{m!} $ for some $a_1, \dots , a_M \in \mathbb{R}$ and $M \in \mathbb{Z}_{>0}$, repeated applications of integration by parts yield
\begin{equation*}
	\int_{X^{n+1}} e^{-t \sum_{k=1}^n \mathsf{d}(x_{k}, x_{k+1})} \mathrm{d} \mu^{n+1} = \int_0^{\infty} e^{- t l } \mu_l^n (P_n^l(X) ) \mathrm{d} l  = \sum_{m=1}^{M} \frac{a_{m}}{t^{m+1}},
\end{equation*}
which is a rational function in $t$ and there are no terms decaying with order $e^{-ct}$ for some $c >0$. A similar conclusion holds when $\mu_l^n (P^l_n (X))$ is of the form $e^{- \alpha l } \left( \sum_{m=1}^{M} a_{m} \frac{l^m}{m!} \right)$, for some $\alpha >0$, which may happen when $\mu$ is a weighted volume form with an exponentially decaying weight, for which we have $\int_{X^{n+1}} e^{-t \sum_{k=1}^n \mathsf{d}(x_{k}, x_{k+1})} \mathrm{d} \mu^{n+1} = \sum_{m=1}^{M} a_{m} (t+ \alpha )^{-m-1}$, which is again a rational function in $t$.

In general, even when $\mu_l^n (P^l_n (X))$ is smooth by Lemma \ref{lmuqgdvlsm}, it does not admit such simplified forms. However, the above argument using the integration by parts does hold as long as $\mu_l^n (P^l_n (X))$ is smooth. The next proposition makes this claim more precise, and shows that the $N$-th partial magnitude for this space is a rational function in $t$ up to remainder terms in the Taylor expansion.

\begin{proposition}
Let $(X,g)$ be a complete Riemannian manifold which is uniquely geodesic. Let $\mu$ be a smooth volume form on $X$ with the following properties:
\begin{enumerate}
	\item $\mu(X)$ is finite,
	\item $\mu_l^n (P^l_n (X))$ is differentiable infinitely many times for all $l > 0$, with the Taylor expansion around the origin as
	\begin{equation*}
		\mu_l^n (P^l_n (X)) = \sum_{m=1}^M a_{n,m} \frac{l^m}{m!}+ r_{n,M}(l)
	\end{equation*}
	where $r_{n,M}$ is the remainder term,
	\item for any $1 \le n \le N$ and $m \in \mathbb{Z}_{>0}$, there exists $C_1 = C_1 (n,m) >0$ such that
	\begin{equation*}
		\sup_{ l \ge 0} \left| \frac{\mathrm{d}^m}{\mathrm{d} l^m } \mu_l^n (P^l_n (X)) \right| \le C_1,
	\end{equation*}
\end{enumerate}
all of which are satisfied when $\mu$ is a smooth weighted volume form which decays rapidly as $l \to \infty$. Then, for any $N,M \in \mathbb{Z}_{>0}$ and any $s>0$, there exist constants $C >0$ and $t_0 >0$, both depending only on $M, N, s, \max_{n=1, \dots , N} \{ C_1 (n,M+1) \}$, such that
\begin{equation*}
	\left|\mathrm{Mag} (X , t \mathsf{d} , \mu ; N )  - \mu (X) - \sum_{m=1}^M \frac{\alpha_m}{t^{m+1}}  - \int_0^{\infty} e^{-tl} r_M (l) \mathrm{d} l \right| \le C e^{-ts/2}
\end{equation*}
holds uniformly for all $t \ge t_0$, where we defined
\begin{equation*}
	\alpha_m := \sum_{n=1}^N (-1)^n a_{n,m} , \quad r_M (l) := \sum_{n=1}^N (-1)^n r_{n,M}(l) .
\end{equation*}
\end{proposition}

\begin{proof}
	We first note that the second hypothesis is always satisfied by Lemma \ref{lmuqgdvlsm}. Thus
	\begin{equation*}
		\phi_n (l):= \mu_l^n (P^l_n (X)) 
	\end{equation*}
	is a smooth function in $l$ for any $l >0$, where we also have $\phi_n (0) =0$. By Lemma \ref{lmdcvfdml} and Theorem \ref{thmggrmmf}, we have
	\begin{equation*}
		\int_{\bm{x} \in X^{n+1}}  e^{- t \sum_{k=1}^n \mathsf{d} (x_{k-1}, x_k)} \mathrm{d} \mu^{n+1} = \int_0^{\infty} e^{-t l } \phi_n (l) \mathrm{d} l .
	\end{equation*}

	We repeat the proof of Watson's lemma (see \cite[Proposition 2.1]{Miller}, and also \cite[\S 5]{Willerton}) to evaluate this integral. For any $s >0$ and any $t>0$, we have
	\begin{equation*}
		\left| \int_s^{\infty} e^{-tl} \phi_n (l) \mathrm{d}l \right| \le e^{-ts} \int_{0}^{\infty} \phi_n (l) \mathrm{d} l = e^{-ts} \mu^{n+1} (X^{n+1}) = e^{-ts} \mu (X)^{n+1}.
	\end{equation*}
	We use the Taylor series for $\phi_n (l)$ with the remainder term
	\begin{equation*}
		\phi_n (l) = \sum_{m=1}^M \frac{\phi_n^{(m)} (0)}{m!} l^m + r_{n,M}(l)
	\end{equation*}
	which is valid for $l \in [0,s]$, where $r_{n,M}(l)$ is the remainder term, which can be bounded by e.g.
	\begin{equation*}
		\left| r_{n,M}(l) \right| \le \sup_{0 \le \tau < s} \left| \phi_n^{(M+1)} (\tau ) \right| \frac{l^{M+1}}{(M+1)!}.
	\end{equation*}
	We thus get
	\begin{equation*}
		\int_0^s e^{-tl} \phi_n (l) \mathrm{d} l = \sum_{m=1}^M \frac{\phi_n^{(m)} (0)}{m!} \int_0^s e^{-tl} l^m \mathrm{d} l + \int_0^s e^{-tl} r_{n,M} (l) \mathrm{d} l,
	\end{equation*}
	where the first integral can be decomposed as
	\begin{equation*}
		\int_0^s e^{-tl} l^m \mathrm{d} l = \int_0^{\infty} e^{-tl} l^m \mathrm{d} l - \int_{s}^{\infty} e^{-tl} l^m \mathrm{d} l = \frac{m!}{t^{m+1}} - \int_{s}^{\infty} e^{-tl} l^m \mathrm{d} l
	\end{equation*}
	with the bound on the second term by Cauchy--Schwarz as
	\begin{equation*}
		\left| \int_{s}^{\infty} e^{-tl} l^m \mathrm{d} l \right| \le e^{-tl/2} \sqrt{\int_0^{\infty} e^{-tl}l^{2m} \mathrm{d} l} = e^{-ts/2} \frac{\sqrt{(2m)!}}{t^{m+1/2}}.
	\end{equation*}
	We thus get
	\begin{equation*}
		\int_0^s e^{-tl} \phi_n (l) \mathrm{d} l - \sum_{m=1}^M \frac{\phi_n^{(m)} (0)}{t^{m+1}}  -   \int_0^s e^{-tl} r_{n,M} \mathrm{d} l = - \sum_{m=1}^M \frac{\phi_n^{(m)} (0)}{m!} \int_s^{\infty} e^{-tl} l^m \mathrm{d} l
	\end{equation*}
	for any $s>0$ and any $t >0$. The above estimates yield
	\begin{align*}
		&\left| \int_0^{\infty} e^{-tl} \phi_n (l) \mathrm{d} l - \sum_{m=1}^M \frac{\phi_n^{(m)} (0)}{t^{m+1}}  - \int_0^{\infty} e^{-tl} r_{n,M} (l) \mathrm{d} l \right| \\
		&\le \left| \sum_{m=1}^M \frac{\phi_n^{(m)} (0)}{m!}  \int_s^{\infty} e^{-tl} l^m \mathrm{d} l \right| + \left| \int_s^{\infty} e^{-tl} \phi_n (l) \mathrm{d}l \right| + \int_{s}^{\infty} e^{-tl} \left| r_{n,M} (l)  \right| \mathrm{d} l \\ 
		&\le e^{-ts/2} \left(  \sum_{m=1}^M \left| \frac{ \phi_n^{(m)} (0)}{m!} \right| \frac{\sqrt{(2m)!}}{t^{m+1/2}} +  \mu (X)^{n+1} + \sup_{0 \le \tau < s} \left| \phi_n^{(M+1)} (\tau ) \right| \frac{\sqrt{(2(M+1))!}}{t^{M+3/2}} \right),
	\end{align*}
	where the terms in the bracket can be bounded by a constant depending only on $M$, $n$, $s$, and $\max_{n=1, \dots , N} \{ C_1 (n,M+1) \}$. We get the required result by summing up the contributions from $n=1, \dots, N$.
\end{proof}

This theorem says, up to the remainder term $\int_0^{\infty} e^{-tl} r_M (l) \mathrm{d} l$ which may be difficult to evaluate in concrete examples, the asymptotic expansion of the partial magnitude function admits no term that decay at order $e^{-ct}$ for any constant $c>0$ when $(X,g)$ is uniquely geodesic. Compared to other examples with a non-trivial cut locus (computed in \S \ref{scexmfdcl}), this property seems to correspond to the results due to Asao \cite{Asao} and Gomi \cite{Gomi} which states the magnitude homology is trivial for uniquely geodesic spaces. For the moment, we verify this phenomenon in the following two examples for small values of $N$.

\begin{example}
Consider $(\mathbb{R} , | \cdot | , e^{-x^2})$, with $N=1$. First note
\begin{align*}
	&\int_{- \infty}^{\infty} \int_{- \infty}^{\infty} e^{-t |x-y|} e^{-x^2 -y^2} \mathrm{d} x \mathrm{d} y \\
	&= \int_{- \infty}^{\infty} \left( \int_{y}^{\infty} e^{-t (x-y)} e^{-x^2} \mathrm{d} x \right) e^{-y^2} \mathrm{d} y  +  \int_{- \infty}^{\infty} \left( \int_{- \infty}^{y} e^{-t (y-x)} e^{-x^2} \mathrm{d} x \right) e^{-y^2} \mathrm{d} y .
\end{align*}
Setting $l = |x-y|$, which implies $x=l+y$ when $x>y$ and $x = y-l$ when $x \le y$, we compute the above integrals by change of variables and Fubini's theorem as
\begin{align*}
	&\int_{- \infty}^{\infty} \left( \int_{0}^{\infty} e^{-tl} e^{-(l+y)^2} \mathrm{d} l \right) e^{-y^2} \mathrm{d} y  +  \int_{- \infty}^{\infty} \left( \int_{0}^{\infty} e^{-t l} e^{-(y-l)^2} \mathrm{d} l \right) e^{-y^2} \mathrm{d} y \\
	&=\int_0^{\infty} e^{-tl} \left(  \int_{- \infty}^{\infty}  e^{-(l+y)^2-y^2} + e^{-(y-l)^2-y^2} \mathrm{d} y  \right) \mathrm{d} l \\
	&= \int_0^{\infty} e^{-tl} \left( \sqrt{2 \pi } e^{-l^2/2} \right) \mathrm{d} l \\
	&= \pi e^{-t^2/2},
\end{align*}
which decays exponentially in $t^2$, and not a constant multiple of $t$. The computation above also implies $\mu_1^l (\mathbb{R}) = \sqrt{2 \pi } e^{-l^2/2}$.

We can proceed to compute the case $N \ge 2$, but the computation quickly becomes difficult. Note that the integral in question 
\begin{equation*}
	\int_{- \infty}^{\infty} \int_{- \infty}^{\infty} \int_{- \infty}^{\infty} e^{-t |x-y|-t|y-z|} e^{-x^2 -y^2-z^2} \mathrm{d} x \mathrm{d} y \mathrm{d} z
\end{equation*}
is not symmetric in $x,y,z$, and we cannot naively continue by induction from the case $N=1$. We can still have some formula for the case $N=2$. Indeed, setting $s_1 = |x-y|$ and $s_2 = |y-z|$, we compute as above to find
\begin{align*}
	&\int_{- \infty}^{\infty} \int_{- \infty}^{\infty} \int_{- \infty}^{\infty} e^{-t |x-y|-t|y-z|} e^{-x^2 -y^2-z^2} \mathrm{d} x \mathrm{d} y \mathrm{d} z \\
	&= \int_{- \infty}^{\infty} e^{-x^2}   \left( \int_{- \infty}^{\infty} e^{-t|x-y|} \left( \int_{0}^{\infty} e^{-ts_2} e^{-(y-s_2)^2} \mathrm{d} s_2 \right) e^{-y^2} \mathrm{d} y \right) \mathrm{d}x \\
	&\quad + \int_{- \infty}^{\infty} e^{-x^2} \left( \int_{- \infty}^{\infty} e^{-t|x-y|}  \left( \int_{0}^{\infty} e^{ -t s_2} e^{-(y+ s_2)^2} \mathrm{d} s_2 \right) e^{-y^2} \mathrm{d} y \right) \mathrm{d}x
\end{align*}
which can be decomposed into the following integrals
\begin{align*}
	&\int_{- \infty}^{\infty} e^{-x^2} \left(  \int_{0}^{\infty} e^{-t s_1} \left( \int_{0}^{\infty} e^{-ts_2} e^{-(s_1+x-s_2)^2} \mathrm{d} s_2 \right) e^{-(s_1+x)^2} \mathrm{d} s_1 \right) \mathrm{d}x \\
	&\quad + \int_{- \infty}^{\infty} e^{-x^2} \left(  \int_{0}^{\infty} e^{-ts_1} \left( \int_{0}^{\infty} e^{-ts_2} e^{-(x-s_1-s_2)^2} \mathrm{d} s_2 \right) e^{-(x-s_1)^2} \mathrm{d} s_1 \right) \mathrm{d}x \\
	&\quad + \int_{- \infty}^{\infty} e^{-x^2} \left( \int_{0}^{\infty} e^{-t s_1}  \left( \int_{0}^{\infty} e^{ -t s_2} e^{-(x+s_1+ s_2)^2} \mathrm{d} s_2 \right) e^{-(x+s_1)^2} \mathrm{d} s_1 \right) \mathrm{d}x \\
	&\quad + \int_{- \infty}^{\infty} e^{-x^2} \left( \int_{0}^{\infty} e^{-t s_1}  \left( \int_{0}^{\infty} e^{ -t s_2} e^{-(x-s_1+ s_2)^2} \mathrm{d} s_2 \right) e^{-(x-s_1)^2} \mathrm{d} s_1 \right) \mathrm{d}x,
\end{align*}
to get
\begin{align*}
	&\int_{- \infty}^{\infty} \int_{- \infty}^{\infty} \int_{- \infty}^{\infty} e^{-t |x-y|-t|y-z|} e^{-x^2 -y^2-z^2} \mathrm{d} x \mathrm{d} y \mathrm{d} z \\
	&= 4\sqrt{\frac{\pi}{3}} \int_{0}^{\infty} \int_{0}^{\infty}  e^{-t(s_1 + s_2)} e^{-(7s^2_1 +4s_2^2)/3} \cosh \left( \frac{10s_1s_2}{3} \right)   \mathrm{d} s_1 \mathrm{d} s_2 .
\end{align*}
Performing the integral above requires the use of error functions.
\end{example}

\begin{example}
Consider $(\mathbb{R} , | \cdot | , e^{-|x|})$, with $N=1$. We compute the first partial magnitude for $(\mathbb{R} , | \cdot | , e^{-|x|})$; this volume form is not smooth, but decays rapidly and puts no mass on non-proper chains. Direct computation yields
\begin{align*}
	&\int_{- \infty}^{\infty} \left( \int_{- \infty}^{\infty} e^{-t |x-y|} e^{-|x|} e^{-|y|} \mathrm{d} x \right) \mathrm{d} y \\
	&=\int_{- \infty}^{0} \left( \int_{- \infty}^{y} e^{-t (y-x)} e^{x} e^{y} \mathrm{d} x \right) \mathrm{d} y + \int_{0}^{\infty} \left( \int_{- \infty}^{0} e^{-t (y-x)} e^{x} e^{- y} \mathrm{d} x \right) \mathrm{d} y \\
	&\quad + \int_{0}^{\infty} \left( \int_{0}^{y} e^{-t (y-x)} e^{-x} e^{-y} \mathrm{d} x \right) \mathrm{d} y + \int_{- \infty}^{0} \left( \int_{y}^{0} e^{-t (x-y)} e^{x} e^{y} \mathrm{d} x \right) \mathrm{d} y \\
	&\quad + \int_{- \infty}^{0} \left( \int_{0}^{\infty} e^{-t (x-y)} e^{-x} e^{y} \mathrm{d} x \right) \mathrm{d} y + \int_{0}^{\infty} \left( \int_{y}^{\infty} e^{-t (x-y)} e^{-x} e^{-y} \mathrm{d} x \right) \mathrm{d} y \\
	&=\frac{2(t+2)}{(t+1)^2} ,
\end{align*}
and hence
\begin{equation*}
	\mathrm{Mag} ( \mathbb{R} , t | \cdot | , e^{-|x|} ; 1 ) = 2 - \frac{2(t+2)}{(t+1)^2}.
\end{equation*}
\end{example}

\subsection{Examples from compact Riemannian manifolds} \label{scexmfdcl}

Let $(X,g)$ be a compact Riemannian manifold. By Theorem \ref{thmggrmmf}, it suffices to compute
\begin{equation*}
	\int_{X^{n+1}} e^{-\sum_{k=1}^n \mathsf{d}(x_{k}, x_{k+1})} \mathrm{d} \mu^{n+1} = \int_0^{\infty} e^{- l} \mu^n_l (P^l_n (X)) \mathrm{d} l
\end{equation*}
for all $n \in \mathbb{Z}_{>0}$ to compute the $\mu$-magnitude. The hard part is to compute $\mu^n_l (P^l_n (X))$. We do so explicitly for various examples for small $n$: we treat $S^1$ and $S^2$ with the standard round metric, 2-torus with the standard flat metric. These examples show that the computation quickly becomes very difficult when $n$ is large. For closed bounded intervals in $\mathbb{R}$ we can compute the integral above for any $n$ (see Lemma \ref{lmmgcbinv}), but even in this case we do not have a closed formula (we only get a recursion relationship).

The difficulty in computing the integral above is that the distance function $\mathsf{d}$ is not differentiable on the cut locus, which necessitates a case-by-case treatment depending on the value of $l$. It leads to the following interesting phenomenon: for a circle $S^1_r$ of radius $r$ (and similarly for $S^2$ and 2-tori), the integrals
\begin{equation*}
	\int_0^{\infty} e^{- tl} \mu^n_l (P^l_n (S^1_r)) \mathrm{d} l,
\end{equation*}
when expanded in $t$, contain a term of order $e^{- \pi r t}$ for any large $n$. Here $\pi r$ is the diameter of $S_r^1$ with respect to the round metric, but it can be also interpreted as the injectivity radius. Indeed, close inspection of the computation reveals that this terms comes from cut locus, where the distance function fails to be differentiable. These examples seems to indicate that the existence of cut locus (i.e.~non-uniqueness of geodesics) is connected to the $\mu$-magnitude, in a subtle but non-trivial manner. While we work only on a limited number of examples, the computation seems to suggest that this phenomenon happens more generally.

Note that all the examples we discuss below are homogeneous manifolds. In general, the injectivity radius depends on the point in the manifold, but the transitivity of the action of the isometry group implies that it does not depend on the point.

\subsubsection{2-torus with $N=1$}
Let $T^2$ be a 2-torus, given by identifying parallel edges of a square (embedded in $\mathbb{R}^2$, with coordinates $x,y$), with a flat metric $g$ and the associated volume form $\mu  = \mathrm{d} \mathrm{vol}_g$. Let $(0,0), (1,0), (0,1), (1,1) \in \mathbb{R}^2$ be the vertices of the square, and we fix a basepoint $y$ at the centre $(1/2 , 1/2)$ of the square (we may move any point to this basepoint by an isometry of $T^2$). For this torus, the diameter is $R := 1/ \sqrt{2}$ and the injectivity radius is $\rho := 1/2$.

We compute the partial $\mu$-magnitude up to order $N=1$. The first term is the volume, which is clearly $\mu (T^2 ) =1$. The term for $N=1$ is given by
\begin{equation*}
	\int_{(x_0, x_1) \in T^2 \times T^2} e^{- t \mathsf{d} (x_0, x_1)} \mathrm{d} \mu^2 = \mu (T^2) \int_{x_1 \in T^2} e^{- t \mathsf{d} (y, x_1)} \mathrm{d} \mu
\end{equation*}
where we moved $x_0$ to the basepoint $y$ by an isometry. We now pass to the polar coordinates around $y$, to write
\begin{equation*}
	\mathrm{d} \mu = \mathrm{d} x \mathrm{d} y = l \mathrm{d} l \mathrm{d} \theta ,
\end{equation*}
where we note that $l = \mathsf{d} (y, x_1)$ is the radial direction which we integrate from $0$ to $R$. Thus the above integral is effectively computing the area of a square (with a weight by $e^{- t \mathsf{d} (y, x_1)}$) in terms of the polar coordinates. We get
\begin{equation*}
	\int_{T^2} e^{- t \mathsf{d} (y, x_1)} \mathrm{d} \mu (x_1) = \int_0^{R} e^{-tl} F(l) \mathrm{d} l
\end{equation*}
where
\begin{equation*}
	F(l) := \begin{cases}
		2 \pi l &\quad ( 0 < l \le \rho) \\
		l \left( 2 \pi - 8 \arccos \left( \frac{1}{2l} \right) \right) &\quad ( \rho \le l  \le R )
	\end{cases}
\end{equation*}
is the surface volume of the level set $\{ x_1 \in T^2 \mid \mathsf{d} (y, x_1) = l\}$. This function is continuous but not differentiable at $l= \rho$, as illustrated in Figure \ref{figclcttr}, in which $F(l)$ is the length of the solid curve.

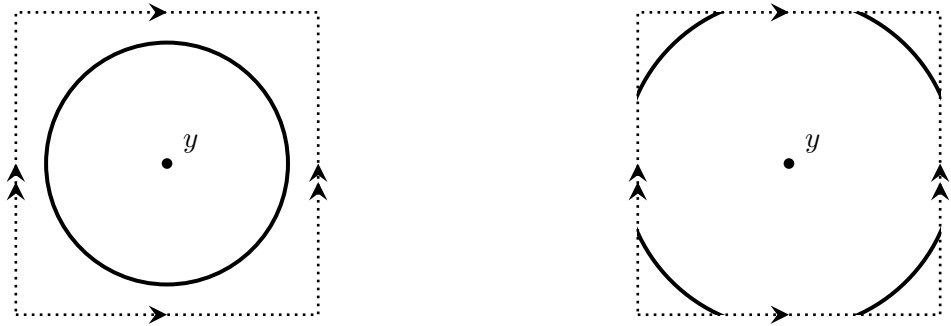
\begin{figure}[h!]
\begin{minipage}[b]{0.49\columnwidth}
\centering
\begin{tikzpicture}[scale=4, 
    >=stealth, 
    mark_arrow/.style={
        decorate,
        decoration={markings, mark=at position 0.5 with {\arrow[scale=1.5]{>}}} 
    },
    mark_arrow2/.style={
        decorate,
        decoration={markings, mark=at position 0.5 with {\arrow[scale=1.5]{>>}}} 
    }] 

    \draw[line width=1pt, dotted] (0, 0) -- (1, 0) -- (1, 1) -- (0, 1) -- cycle;

    \draw[line width=1.5pt] (0.5, 0.5) circle (0.4);

    
    \draw[mark_arrow, line width=1pt] (0, 0) -- (1, 0); 
    
    \draw[mark_arrow, line width=1pt] (0, 1) -- (1, 1);

    \draw[mark_arrow2, line width=1pt] (0, 0) -- (0, 1); 
    \draw[mark_arrow2, line width=1pt] (1, 0) -- (1, 1); 
    
    \fill (0.5, 0.5) circle (0.5pt) node[above right, xshift=2pt, font=\small] {$y$};

\end{tikzpicture}
\end{minipage}
\begin{minipage}[b]{0.49\columnwidth}
\centering
\begin{tikzpicture}[scale=4, 
    >=stealth, 
    mark_arrow/.style={
        decorate,
        decoration={markings, mark=at position 0.5 with {\arrow[scale=1.5]{>}}} 
    },
    mark_arrow2/.style={
        decorate,
        decoration={markings, mark=at position 0.5 with {\arrow[scale=1.5]{>>}}} 
    }] 

    \begin{scope}
        \clip (0, 0) rectangle (1, 1);
        
        \draw[line width=1.5pt] (0.5, 0.5) circle (0.55);
    \end{scope}

    \draw[line width=1pt, black, dotted] (0, 0) -- (1, 0) -- (1, 1) -- (0, 1) -- cycle;

    
    \draw[mark_arrow, line width=1pt] (0, 0) -- (1, 0); 
    
    \draw[mark_arrow, line width=1pt] (0, 1) -- (1, 1);

    \draw[mark_arrow2, line width=1pt] (0, 0) -- (0, 1); 
    \draw[mark_arrow2, line width=1pt] (1, 0) -- (1, 1); 

    \fill (0.5, 0.5) circle (0.5pt) node[above right, xshift=2pt, font=\small] {$y$};

\end{tikzpicture}
\end{minipage}
\caption{The level set $\mathsf{d} (y, x_1) = l$, for $l < \rho$ (left) and $l > \rho$ (right).}
\label{figclcttr}
\end{figure}

We compute the above integral as
\begin{equation*}
\int_0^{R} e^{-tl} F(l) \mathrm{d} l =   \frac{2 \pi}{t^2} - 8  \int_{\rho}^{R} e^{-t l} l \arccos \left( \frac{1}{2l} \right) \mathrm{d} l  - 2 \pi \frac{R t+1}{t^2} e^{-t R} ,
\end{equation*}
where the decay rate of the middle term can be evaluated by
\begin{align*}
	\frac{e^{- \rho t}}{t^2} \left( 1 + \frac{t(2 \rho -1)}{2} - \frac{(2 -t+ 2 R t)e^{- t (R - \rho )}}{2} \right) &\le \int_{\rho}^{R} e^{-t l } l \arccos \left( \frac{1}{2l} \right) \mathrm{d} l \\
	&\le \frac{e^{- \rho t}}{t^2} \frac{\pi}{2} \left( 1+\rho t  - (1+ R t)e^{- t (R - \rho )} \right)
\end{align*}
by using $\frac{2l-1}{2} \le l \arccos ( \frac{1}{2l} ) \le \frac{\pi l}{2}$ which follows from a crude estimate $1-x \le \cos x$. While the above is just an inequality, it does prove that the exponential decay rate of the integral is of order $e^{- \rho t}$. Thus, the first partial magnitude function can be computed as
\begin{equation*}
	\mathrm{Mag} (T^2 , t \mathsf{d} , \mu ;1) = 1 - \frac{2 \pi}{t^2} + \left( \text{terms of order between } \frac{e^{- \rho t}}{t^2} \text{ and } \frac{e^{- \rho t}}{t} \right) + 2 \pi \frac{R t+1}{t^2} e^{-t R} .
\end{equation*}
In this formula, there are two terms that decay exponentially, which are $e^{- t \rho}$ and $e^{-t R}$. These correspond to the lengths of geodesics for which the number of geodesics from $y \in T^2$ jumps. Indeed, the number of geodesics between $p \in T^2$ and $y$ is one as as long as the length between $p$ and $y$ is less than the injectivity radius $\rho$. When the length reaches $\rho$, the number of geodesics jumps to 2. Furthermore, when $p$ reaches the point $q$ which achieves $\mathsf{d} (y,q) = R$ (the diameter), the number of geodesics jumps to 4 (see Figure \ref{figspqjp24}).

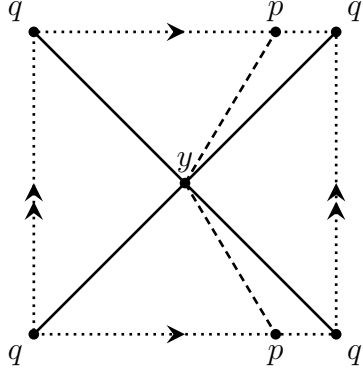
\begin{figure}[h!]
\centering
\begin{tikzpicture}[scale=4,
>=stealth, 
    mark_arrow/.style={
        decorate,
        decoration={markings, mark=at position 0.5 with {\arrow[scale=1.5]{>}}} 
    },
    mark_arrow2/.style={
        decorate,
        decoration={markings, mark=at position 0.5 with {\arrow[scale=1.5]{>>}}} 
    }]

\coordinate (A) at (0,0);
\coordinate (B) at (1,0);
\coordinate (C) at (1,1); 
\coordinate (D) at (0,1); 

\coordinate (P) at (0.5, 0.5); 
\coordinate (Q) at (0.8, 1);
\coordinate (R) at (0.8, 0);

\draw[line width=1pt, dotted] (A) -- (B) -- (C) -- (D) -- cycle;
   \draw[mark_arrow, line width=1pt] (0, 0) -- (1, 0); 
    
    \draw[mark_arrow, line width=1pt] (0, 1) -- (1, 1);

    \draw[mark_arrow2, line width=1pt] (0, 0) -- (0, 1); 
    \draw[mark_arrow2, line width=1pt] (1, 0) -- (1, 1); 

\draw[line width=1pt] (P) -- (A); 
\draw[line width=1pt] (P) -- (D); 
\draw[line width=1pt] (P) -- (B); 
\draw[line width=1pt] (P) -- (C); 

\draw[line width=1pt, densely dashed] (P) -- (Q);
\draw[line width=1pt, densely dashed] (P) -- (R);

\fill (P) circle (0.5pt) node[above] {$y$};
\fill (Q) circle (0.5pt) node[above] {$p$};
\fill (R) circle (0.5pt) node[below] {$p$};

\fill (A) circle (0.5pt) node[below left] {$q$};
\fill (B) circle (0.5pt) node[below right] {$q$};
\fill (C) circle (0.5pt) node[above right] {$q$};
\fill (D) circle (0.5pt) node[above left] {$q$};

\end{tikzpicture}
\caption{The number of geodesics jumps from 2 to 4 when $p$ reaches $q$.}
\label{figspqjp24}
\end{figure}

Recall also that there are at least two geodesics between two points that attain the diameter \cite[Exercise 2.118]{GHL04}.

\subsubsection{Circle with $N=3$}
	Let $X$ be the circle $S^1_r$ of radius $r$, with respect to the standard round metric. As in the proof of Lemma \ref{lmuqgdvlsm}, we write
	\begin{equation*}
		P_n^l (S^1_r) = \bigsqcup_{s_1+ \cdots + s_n=l} \bigsqcup_{x_0 \in S^1_r}  \{ (x_0 , \dots , x_n ) \in P_n(S^1_r) \mid x_j \in S^0_{s_j} (x_{j-1}) , \; j=1 , \dots , n \} ,
	\end{equation*}
	where $S^0_{s_j} (x_{j-1})$ is a geodesic sphere of radius $s_j$ centred at $x_{j-1}$, which just consists of two points, and $s_1 , \dots , s_n \in [0 , \pi r]$ satisfies $s_1 + \cdots + s_n = l$. A crucial point here is that each $s_k$ can never exceed $\pi r$, which is the diameter of $S^0_r$.
	
	Consider a direct product of $n+1$ circles $(S_r^1)^{n+1}$, and write $\theta_{j} \in [ 0, 2 \pi )$ for the coordinate for the $j$-th copy of $S_r^1$. Then we have
	\begin{equation*}
		\mathrm{d} \mu^{n+1} = r^{n+1} \mathrm{d} \theta_{0} \wedge \cdots \wedge \mathrm{d} \theta_{n} .
	\end{equation*}
	Substituting in $\mathrm{d} \theta_{1} = \frac{1}{r} \mathrm{d} l - \sum_{j=2}^{n} \mathrm{d} \theta_{j}$ and re-writing $r \theta_{j} = s_j$ as above, we find
	\begin{equation*}
		\mathrm{d} \mu^{n+1} = r \mathrm{d} \theta_{0} \wedge \mathrm{d} l \wedge  \mathrm{d} s_2 \wedge  \cdots \wedge \mathrm{d} s_n 
	\end{equation*}
	and hence
	\begin{equation*}
		\mathrm{d} \mu^{n+1}_l = r \mathrm{d} \theta_{0}  \cdot \mathrm{d} s_2 \cdots \mathrm{d} s_n
	\end{equation*}
	as a measure (as opposed to a volume form which could have a different sign). We thus have
	\begin{equation*}
		\mu^n_l (P^l_n (S^1_r)) = 2^{n+1} \pi r  \int_{0 \le l-s_2 - \dots - s_n \le \pi r, \; s_2 , \dots , s_n \in [0, \pi r]}  \mathrm{d} s_2 \cdots \mathrm{d} s_n 
	\end{equation*}
	where we note the factor $2^{n+1} = |S^0|^{n+1}$.

\begin{figure}[h!]
\centering
\begin{tikzpicture}[scale=3.5] 

    \draw[->, thin] (-0.1, 0) -- (1.2, 0) node[right] {$s_1$};
    \draw[->, thin] (0, -0.1) -- (0, 1.2) node[above] {$s_2$};

	\node at (1, -0.1)  {$\pi r$};
	\node at (-0.1, 1) {$\pi r$};
    
    \fill[gray!10] (0, 0) -- (1, 0) -- (1, 1) -- (0, 1) -- cycle;

    \draw[thin, black] (0, 0) -- (1, 0) -- (1, 1) -- (0, 1) -- cycle;

    \draw[line width=1.5pt] (0, 0.5) -- (0.5, 0);
    \node at (-0.5,0.3) {$l=s_1+s_2 < \pi r$};
    
    \draw[line width=1.5pt, dashed] (0, 1) -- (1, 0);

    \draw[line width=1.5pt] (0.5, 1) -- (1, 0.5);

    \node at (1.5, 0.7) {$l=s_1+s_2 > \pi r$};

\end{tikzpicture}
\caption{The length of the solid line changes non-differentiably at $l= \pi r$.}
\label{figsqcptcd}
\end{figure}
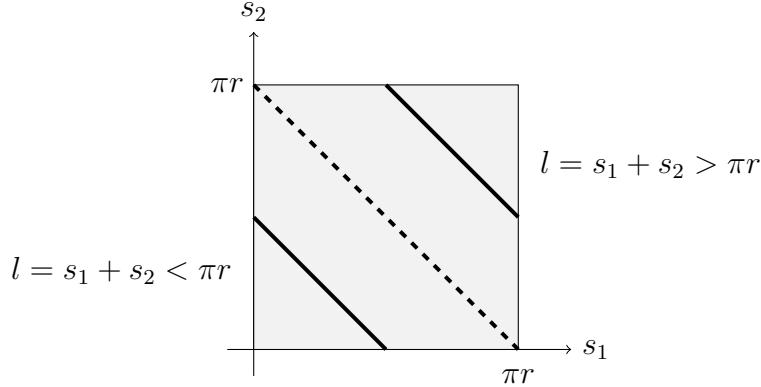

When $n=2$, the range of the integral above is precisely the intersection of the line $x+y=l$ and the unit cube as described in Figure \ref{figsqcptcd}, which  should be compared to Figure \ref{figsqhdmnc}. This area (or length) varies continuously in $l$, but not smoothly, since this length is clearly not differentiable at $l= \pi r$.

We compute this integral explicitly. When $n=1$, we simply get
\begin{equation*}
	\int_{S^1_r \times S^1_r} e^{- t \mathsf{d} (x_0, x_1)} r^2 \mathrm{d} \theta^{(0)} \mathrm{d} \theta^{(1)} = 4 \pi r \int_{0}^{\pi r} e^{-tl} \mathrm{d} l = 4 \pi r \frac{1 - e^{- \pi rt}}{t} = \frac{4 \pi r}{t} - \frac{4 \pi r}{t} e^{- \pi r t}.
\end{equation*}

When $n=2$, noting that $l$ is at most $2 \pi r$ in that case, we have
\begin{equation*}
	\mu^2_l (P^l_2 (S^1_r)) = 8 \pi r \int_{\max \{0, l - \pi r\}}^{\min \{ \pi r,l \}}  \mathrm{d} s_2 = \begin{cases}
		8 \pi r l &\quad (0 \le l \le \pi r ) \\
		8 \pi r (2 \pi r -l) &\quad ( \pi r \le l \le 2 \pi r)
	\end{cases}
\end{equation*}
which is the length of the solid line in Figure \ref{figsqcptcd} (up to a factor of $\sqrt{2}$ coming from the normalisation of $\mathrm{d} s$). Thus
\begin{align*}
	\int_0^{2 \pi r} e^{- tl} \mu^2_l (P^l_2 (S^1_r)) \mathrm{d} l &= 8 \pi r \int_0^{\pi r} e^{-tl} l \mathrm{d} l + 8 \pi r \int_{\pi r}^{2 \pi r} e^{-tl} (2 \pi r -l) \mathrm{d} l \\
	&= \frac{8 \pi r}{t^2} ( 1-2e^{- \pi rt} +e^{- 2 \pi r t } ). 
\end{align*}

When $n=3$, $l$ is at most $3 \pi r$. Noting $l - 2 \pi r \le s_2 = l-s_1-s_3 \le l$, we have
\begin{align*}
		\mu^3_l (P^l_3 (S^1_r)) &= 16 \pi r \int_{0 \le l-s_2 - s_3 \le \pi r, \; s_2 , s_3 \in [0, \pi r]}  \mathrm{d} s_2 \mathrm{d} s_3 \\
		&= 16 \pi r  \int_{\max \{0, l - 2 \pi r\}}^{\min \{ \pi r,l \}} \left( \int_{\max \{0, l -s_2 - \pi r\}}^{\min \{ \pi r, l - s_2 \}} \mathrm{d} s_3 \right) \mathrm{d} s_2  \\
		&= 16 \pi r  \int_{\max \{0, l - 2 \pi r\}}^{\min \{ \pi r,l \}} \left(  (l-s_2) \mathbbm{1}_{\{  0 \le l - s_2 \le \pi r \}} + (2 \pi r - l+s_2) \mathbbm{1}_{\{  l - s_2 > \pi r \}} \right) \mathrm{d} s_2 ,
\end{align*}
where $\mathbbm{1}$ is the indicator function. Direct computation yields
\begin{equation*}
	\int_{\max \{0, l - 2 \pi r\}}^{\min \{ \pi r,l \}} \left( \int_{\max \{0, l -s_2 - \pi r\}}^{\min \{ \pi r, l - s_2 \}} \mathrm{d} s_3 \right) \mathrm{d} s_2 = \begin{cases}
		l^2/2 &\quad (0 \le l \le \pi r ) \\
		3 \pi r l -l^2 - 3 \pi^2 r^2 /2 &\quad (\pi r \le l \le 2 \pi r ) \\
		( l^2 - 6 \pi r l + 9 \pi^2 r^2)/2 &\quad (2 \pi r \le l \le 3 \pi r )
	\end{cases}
\end{equation*}
which in turn implies
\begin{align*}
	&\int_0^{\infty} e^{- tl} \mu^3_l (P^l_3 (S^1_r)) \mathrm{d} l \\
	&=\pi r \int_0^{ \pi r} e^{- tl} l^2 \mathrm{d} l + \pi r \int_{\pi r}^{2 \pi r} e^{- tl} \left(6 \pi r l -2 l^2 - 3 \pi^2 r^2 \right) \mathrm{d} l + \pi r \int_{2 \pi r}^{3 \pi r} e^{- tl} \left(  l^2 - 6 \pi r l + 9 \pi^2 r^2 \right) \mathrm{d} l \\
	&=\frac{1}{t^3}+\left(\frac{1}{t^3}-\frac{2\pi r}{t^2}\right)e^{-\pi r t}+\left(\frac{12\pi^2 r^2}{t}+\frac{4\pi r}{t^2}-\frac{1}{t^3}\right)e^{-2\pi r t} +\left(-\frac{18\pi^2 r^2}{t}+\frac{6\pi r}{t^2}-\frac{1}{t^3}\right)e^{-3\pi r t}
\end{align*}
As in the case of 2-tori, we have exponentially decaying terms corresponding to the injectivity radius $\pi r$ (which is also the diameter of $S^1$). Furthermore, we have integer multiples of $\pi r$ up to $n=3$. These numbers precisely correspond to the degrees for which the magnitude homology for $S_r^1$ is non-trivial, as computed by Gomi \cite[page 4]{Gomi}.

\subsubsection{2-sphere with $N=2$} \label{ssc2shn2ex}
We proceed similarly for 2-spheres. Let $S^2_r$ be a $2$-sphere of radius $r>0$, endowed with a standard round metric.

As we did for circles, we write
	\begin{equation*}
		P_n^l (S^2_r) = \bigsqcup_{s_1+ \cdots + s_n=l} \bigsqcup_{x_0 \in S^2_r}  \{ (x_0 , \dots , x_n ) \in P_n(S^2_r) \mid x_j \in S^1_{s_j} (x_{j-1}) , \; j=1 , \dots , n \} ,
	\end{equation*}
	where $S^1_{s_j} (x_{j-1})$ is a geodesic sphere of radius $s_j$ centred at $x_{j-1}$, and $s_1 , \dots , s_n \in [0 , \pi r]$ satisfies $s_1 + \cdots + s_n = l$. Again, each $s_k$ can never exceed the diameter $\pi r$ which is also its injectivity radius. We compute the volume form $\mu_l^n$ for general $n$; the following computation easily generalises to spheres of any dimensions, with more complicated formulae. The volume form of $S^2_r$ is given by
	\begin{equation*}
		\mathrm{d}  \mu = r^{2} \sin ( \theta_1 )  \mathrm{d}  \theta_1 \wedge \mathrm{d}  \theta_{2}
	\end{equation*}
	where $\theta_1 , \theta_{2}$ are spherical polar coordinates, with $\theta_1 \in [0, \pi ]$ and $\theta_2  \in [0, 2 \pi )$. Thus
	\begin{equation*}
		\mathrm{d}  \mu^{n+1} = \pm r^{2(n+1)} \prod_{j=1}^{n+1} \sin ( \theta^{(j)}_1 ) d \theta^{(j)}_1 \wedge \cdots \wedge d \theta^{(j)}_{2}
	\end{equation*}
	where the upper indices in brackets correspond to the copies in the product set $(S^2_r)^{n+1}$, and the sign can be chosen to be $+$ by choosing an appropriate orientation for $(S^2_r)^{n+1}$. The constraint $s_1 + \cdots + s_n = l$ gives an equation $\sum_{j=2}^{n+1} \theta^{(j)}_1 = l/r$, defining a real hypersurface in $(S^2_r)^{n+1}$, which is smooth outside a set of measure zero. We thus get 
	\begin{equation*}
		\mathrm{d} \mu^{n+1} = \pm \left( r^{2(n+1)} \prod_{j=0}^{n} \sin ( \theta^{(j)}_1 ) \right) \mathrm{d} \theta^{(0)}_1 \wedge \mathrm{d} \theta^{(0)}_{2} \wedge \cdots \wedge \mathrm{d} \theta^{(n)}_1 \wedge \mathrm{d} \theta^{(n)}_{2} .
	\end{equation*}
	Substituting in $\mathrm{d} \theta_1^{(1)} = \frac{1}{r} \mathrm{d} l - \sum_{j=2}^{n} \mathrm{d} \theta_1^{(j)}$ and re-writing $r \theta_1^{(j)} = s_j$, we find
	\begin{align*}
		\mathrm{d} \mu^{n+1} = &\pm \left( r^{n+2} \sin ( \theta^{(0)}_1 ) \sin \left( \frac{1}{r} \left( l-\sum_{j=2}^n s_j \right) \right) \prod_{j=2}^{n} \sin \left( \frac{s_j}{r} \right) \right) \\
		&\quad \times \mathrm{d} \theta^{(0)}_1 \wedge \mathrm{d} \theta^{(0)}_{2} \wedge \mathrm{d} l \wedge \mathrm{d} \theta^{(1)}_{2} \wedge \mathrm{d} s_2 \wedge \mathrm{d} \theta^{(2)}_{2} \wedge \cdots \wedge \mathrm{d} s_n \wedge \mathrm{d} \theta^{(n)}_{2},
	\end{align*}
	and hence
	\begin{align*}
		\mathrm{d} \mu^{n+1}_l = &\left( r^{n+2} \sin ( \theta^{(0)}_1 ) \sin \left( \frac{1}{r} \left( l-\sum_{j=2}^n s_j \right) \right) \prod_{j=2}^{n} \sin \left( \frac{s_j}{r} \right) \right) \\
		&\quad \times \mathrm{d} \theta^{(0)}_1  \mathrm{d} \theta^{(0)}_{2} \mathrm{d} \theta^{(1)}_{2}  \mathrm{d} \theta^{(2)}_{2} \cdots \mathrm{d} \theta^{(n)}_{2} \cdot \mathrm{d} s_2 \cdots \mathrm{d} s_n
	\end{align*}
	as a measure (or by fixing the orientation). We thus have, noting $\int_0^{\pi} \sin ( \theta^{(0)}_1 ) \mathrm{d} \theta^{(0)}_1 = 2$,
	\begin{align*}
		&\mu^n_l (P^l_n (S^2_r)) \\
		&= 2 (2 \pi)^{n+1} r^{n+2} \int_{0 \le l-s_2 - \dots - s_n \le \pi r, \; s_2 , \dots , s_n \in [0, \pi r]} \sin \left( \frac{1}{r} \left( l-\sum_{j=2}^n s_j \right) \right) \prod_{j=2}^{n} \sin \left( \frac{s_j}{r} \right)  \mathrm{d} s_2 \cdots \mathrm{d} s_n .
	\end{align*}
	The last integral can be equally written by using
	\begin{equation*}
		Q_{n,l}:= \{ (s_1 , \dots , s_n ) \in [0 , \pi r ]^n \mid s_1 + \cdots + s_n = l \} \subset \rl^n,
	\end{equation*}
	as
	\begin{equation*}
		\mu^n_l (P^l_n (S^2_r)) = 2 (2 \pi)^{n+1} r^{n+2} \int_{Q_{n,l}} \sin \left( \frac{s_1}{r} \right) \cdots \sin \left( \frac{s_n}{r} \right) \mathrm{d} s
	\end{equation*}
	where $\mathrm{d} s$ is the $(n-1)$-dimensional Hausdorff measure on $Q_{n,l}$ induced by the Euclidean metric.

For $n=1$, we have
\begin{equation*}
	\int_0^{\infty} e^{-tl} \mu^1_l (P^l_1 (S^2_r)) \mathrm{d} l = 2 (2 \pi)^{2} r^{3} \int_{0}^{\pi r} e^{-tl} \sin \left( \frac{l}{r} \right) \mathrm{d} l = 2 (2 \pi)^{2} r^{4} \frac{1+e^{- \pi r t}}{t^2r^2+1} .
\end{equation*}

For $n=2$, we similarly get
\begin{align*}
	\mu^2_l (P^l_2 (S^2_r)) &= 2 (2 \pi)^{3} r^{4} \int_{\max \{0, l - \pi r\}}^{\min \{ \pi r,l \}} \sin \left( \frac{l-s}{r} \right) \sin \left( \frac{s}{r} \right) \mathrm{d} s \\
	&=\begin{cases}
		(2 \pi)^{3} r^{4} \left( l \cos \left( \frac{l}{r} \right) + r \sin \left( \frac{l}{r} \right) \right) &\quad (0 \le l \le \pi r) \\
		(2 \pi)^{3} r^{4} \left( (2 \pi r -l ) \cos \left( \frac{l}{r} \right) + r \sin \left( \frac{2 \pi r - l}{r} \right) \right)  &\quad ( \pi r \le l \le 2 \pi r)
	\end{cases}
\end{align*}
and hence direct computation yields
\begin{align*}
	& \int_0^{\infty} e^{-tl} \mu^2_l (P^l_2 (S^2_r)) \mathrm{d} l \\
	&= \frac{8 \pi^3 r^4}{1+r^2t^2} \left( \pi r^{2} e^{- \pi rt}+r^{2}(1+e^{- \pi rt})+ \frac{2r^{3}t(1+e^{- \pi rt})+2r^{4}t^{2}(e^{- \pi rt}+e^{- 2 \pi rt})}{1+r^2t^2} \right)
\end{align*}
As in the case of circles, we have exponentially decaying terms, where integer multiples of $\pi r$ appear up to $n=2$ in the exponent.

\subsubsection{Closed bounded intervals} \label{sccbilm}

We finally discuss the line segment $[a,b] \subset \rl$ of length $L := b-a$. We compute the $(\mu , \Gamma^{\mathrm{triv}})$-magnitude with respect to the Lebesgue measure $\mu$, as opposed to the weight measure treated in \S \ref{sccbism}. Note that the Lebesgue measure is the uniform measure in the sense of \cite[Definition 9.2 and Proposition 9.5]{LR21}.

Since $\mu ([a,b]) = b-a =:L$, the magnitude is given by
\begin{equation*}
	\mathrm{Mag} ([a,b] , t \mathsf{d} , \mu ) = L + \sum_{n=1}^{\infty} (-1)^n \int_{[a,b]^{n+1}} e^{- t \sum_{k=1}^n |x_{k-1}-x_k|} \mathrm{d} \mu^{n+1} ,
\end{equation*}
which converges for $t \gg 1$ as proved in Theorem \ref{ppexmgpmlt}.

\begin{remark}
Since $\mu$ is the Lebesgue measure, we have
\begin{equation*}
	\mathrm{Mag} ([a,b] , t \mathsf{d} , \mu ) = \mathrm{Mag} ((a,b) , t \mathsf{d} , \mu ),
\end{equation*}
and hence we may equally compute the magnitude of the open interval, where we note that $(a,b)$ is not complete. This phenomenon happens since $\mu$ has no singular part, and it needs to be compared to the example discussed in \S \ref{sccbism} where the weight measure has a singular part at the boundary.
\end{remark}

We compute
\begin{equation*}
	\int_{[a,b]^{n+1}} e^{- t \sum_{k=1}^n |x_{k-1}-x_k|} \mathrm{d} \mu^{n+1} ,
\end{equation*}
by using auxiliary integrals
\begin{align*}
	\alpha_{n-1,k}(t) &:=\int_{[a,b]^{n-1}} e^{- t \sum_{k=1}^{n-1} |x_{k-1}-x_k|}  \frac{(b-x_{n-1})^k}{k!} e^{-t(b-x_{n-1})} \mathrm{d} \mu^{n-1} \\
	&\quad + \int_{[a,b]^{n-1}} e^{- t \sum_{k=1}^{n-1} |x_{k-1}-x_k|} \frac{(x_{n-1}-a)^k}{k!}e^{-t(x_{n-1}-a)} \mathrm{d} \mu^{n-1}
\end{align*}
defined for each non-negative integer $k$ and $n \ge 2$.

\begin{lemma} \label{lmmgcbinv}
	We have
	\begin{equation} \label{eqlmcsinpm}
		\int_{[a,b]^{n+1}} e^{- t \sum_{k=1}^n |x_{k-1}-x_k|} \mathrm{d} \mu^{n+1} = \frac{2^{n}}{t^{n}}L - \frac{2^n}{t^{n+1}}(1-e^{-tL}) - \sum_{m=1}^{n-1} \frac{2^{m-1}}{t^m} \alpha_{n-m,0} (t)
	\end{equation}
	where $\alpha_{n-m,0} (t)$ can be computed by the recursion relationship
	\begin{equation} \label{eqlmcsinrr}
		\alpha_{n-1,k}(t)= \sum_{m=0}^{k+1} \frac{1}{(2t)^{m}} \alpha_{n-2,k-m+1} (t) - \alpha_{n-2,0} (t) \sum_{k=0}^m \frac{L^{k-m} e^{-tL}}{(2t)^{m+1} (k-m)!} ,
	\end{equation}
	which holds for any $n \ge 3$ and any $k \ge 0$, and
	\begin{align}
		{\alpha}_{1,k} (t) &= - \sum_{m=0}^{k+1} \frac{2e^{-tL}}{(2t)^{m}} \sum_{j=0}^{k-m+1} \frac{L^{k-m+1-j}}{t^{j+1}(k-m+1-j)!} \notag \\
		&\quad  -  \sum_{m=0}^{k} \frac{4e^{-tL} L^{k-m} }{(2t)^{m+2} (k-m)!} ( 1-e^{-tL} ) + \frac{2^{k+2}-2}{(2t)^{k+2}} (1-e^{-tL}) . \label{eqlmcsinic}
	\end{align}
\end{lemma}

\begin{proof}
We first integrate with respect to the last variable $x_n$ to get
\begin{align*}
	&\int_{[a,b]^{n+1}} e^{- t \sum_{k=1}^n |x_{k-1}-x_k|} \mathrm{d} \mu^{n+1} \\
	&= \int_{[a,b]^{n}} e^{- t \sum_{k=1}^{n-1} |x_{k-1}-x_k|} \left( \int^b_{x_{n-1}} e^{- t (x_n - x_{n-1} )} \mathrm{d} x_n + \int_a^{x_{n-1}} e^{- t (x_{n-1}-x_n)} \mathrm{d} x_n \right) \mathrm{d} \mu^{n} \\
	&=\frac{2}{t} \int_{[a,b]^{n}} e^{- t \sum_{k=1}^{n-1} |x_{k-1}-x_k|} \mathrm{d} \mu^{n} - \frac{\alpha_{n-1,0} (t)}{t}  
\end{align*}
Thus, together with the elementary integral
\begin{equation} \label{eqfsmglsl}
	\int_a^b \int_a^b e^{- t |x-y|} \mathrm{d} x \mathrm{d} y = \frac{2L}{t} - \frac{2}{t^2}(1-e^{-tL}),
\end{equation}
we find the claimed equality (\ref{eqlmcsinpm}). It remains to prove the equalities (\ref{eqlmcsinrr}) and (\ref{eqlmcsinic}). We first compute the following integral
\begin{equation*}
	\int_a^b \left( \frac{(b-x_{n-1})^k}{k!} e^{-t|x_{n-2}-x_{n-1}|-t(b-x_{n-1})} + \frac{(x_{n-1}-a)^k}{k!}e^{-t|x_{n-2}-x_{n-1}|-t(x_{n-1}-a)} \right) \mathrm{d} x_{n-1}.
\end{equation*}
By repeated applications of integration by parts, we find
\begin{align*}
	&\int_a^b \frac{(b-x_{n-1})^k}{k!} e^{-t|x_{n-2}-x_{n-1}|-t(b-x_{n-1})}  \mathrm{d} x_{n-1} \\
	&=\int_a^{x_{n-2}} \frac{(b-x_{n-1})^k}{k!} e^{-t(b+x_{n-2})+2tx_{n-1}}  \mathrm{d} x_{n-1} + e^{-t(b-x_{n-2})} \int_{x_{n-2}}^b \frac{(b-x_{n-1})^k}{k!}   \mathrm{d} x_{n-1} \\
	&= \sum_{m=0}^{k} \left( \frac{1}{(2t)^{m+1}} \frac{(b-x_{n-2})^{k-m}}{(k-m)!} e^{-t(b-x_{n-2})} - \frac{1}{(2t)^{m+1}} \frac{L^{k-m}}{(k-m)!} e^{-t(x_{n-2}-a)-tL} \right) \\
	&\quad +  e^{-t(b-x_{n-2})} \frac{(b-x_{n-2})^{k+1}}{(k+1)!} ,
\end{align*}
and similarly
\begin{align*}
	&\int_a^b \frac{(x_{n-1}-a)^k}{k!} e^{-t|x_{n-2}-x_{n-1}|-t(x_{n-1}-a)}  \mathrm{d} x_{n-1} \\
	&= e^{-t(x_{n-2}-a)} \frac{(x_{n-2}-a)^{k+1}}{(k+1)!} \\
	&\quad  + \sum_{m=0}^{k} \left( \frac{1}{(2t)^{m+1}} \frac{(x_{n-2}-a)^{k-m}}{(k-m)!} e^{-t(x_{n-2}-a)} - \frac{1}{(2t)^{m+1}} \frac{L^{k-m}}{(k-m)!} e^{-t(b-x_{n-2})-tL} \right) .
\end{align*}
We thus get
\begin{equation*}
	\alpha_{n-1,k}(t)= \sum_{m=0}^{k+1} \frac{1}{(2t)^{m}} \alpha_{n-2,k-m+1} (t) - \alpha_{n-2,0} (t) \sum_{k=0}^m \frac{L^{k-m} e^{-tL}}{(2t)^{m+1} (k-m)!} 
\end{equation*}
which is the claimed relationship (\ref{eqlmcsinrr}). The initial term $\alpha_{1,k} (t)$ can be computed directly as
\begin{align*}
	\alpha_{1,k} (t) &= \int_a^b   \int_a^b \left(\frac{(b-y)^k}{k!} e^{-t|x-y|-t(b-y)} + \frac{(y-a)^k}{k!}e^{-t|x-y|-t(y-a)} \right) \mathrm{d} y  \mathrm{d} x  \\
	&= \sum_{m=0}^{k+1} \frac{1}{(2t)^{m}} \int_a^b  \left( \frac{(x -a)^{k-m+1}}{(k-m+1)!} e^{-t(x -a)} + \frac{(b-x)^{k-m+1}}{(k-m+1)!} e^{-t(b-x)} \right) \mathrm{d} x \\
	&\quad - \sum_{m=0}^{k} \frac{L^{k-m} e^{-tL}}{(2t)^{m+1} (k-m)!} \int_a^b \left(   e^{-t(b-x)} + e^{-t(x-a)} \right) \mathrm{d} x .
\end{align*}
Noting
\begin{equation*}
	\int_a^b \frac{(x -a)^{k-m+1}}{(k-m+1)!} e^{-t(x -a)} \mathrm{d} x =- \sum_{j=0}^{k-m+1} \frac{L^{k-m+1-j}e^{-tL}}{t^{j+1}(k-m+1-j)!} + \frac{1}{t^{k-m+2}} (1-e^{-tL})
\end{equation*}
and
\begin{equation*}
	\int_a^b  \frac{(b-x)^{k-m+1}}{(k-m+1)!} e^{-t(b-x)} \mathrm{d} x  = - \sum_{j=0}^{k-m+1} \frac{L^{k-m+1-j}e^{-tL}}{t^{j+1}(k-m+1-j)!} + \frac{1}{t^{k-m+2}} (1-e^{-tL})
\end{equation*}
gives the claimed equality (\ref{eqlmcsinic}).
\end{proof}

Setting $L:= b-a$ and recalling the equality (\ref{eqfsmglsl}), the first partial magnitude function can be computed as
\begin{equation*}
	\mathrm{Mag} ([a,b] , t \mathsf{d} , \mathrm{d} x ;1) = L - \frac{2L}{t} + \frac{2}{t^2}(1-e^{-tL}).
\end{equation*}

The expansion above contains a term of order $e^{-tL}$, coming from the boundary of $[a,b]$, while the magnitude homology of $[a,b]$ is trivial since it is (Menger) convex \cite{Gomi,LeiShu21}. Compared to other examples in this section, the phenomenon above seems to indicate that an extra care is necessary to deal with manifolds with boundaries.

\section{Metric spaces with weight measures} \label{scmswms}

\subsection{Magnitude equals volume for a weight measure}

A strategy for extending the magnitude for infinite metric spaces, other than taking the supremum over finite subspaces (\ref{eqdfclmg}), is to consider the volume with respect to the weight measure \cite{Mec13,Willerton}.

\begin{definition}
Let $(X , \mathsf{d})$ be a metric space. A Borel measure $\mu_{\mathrm{w}}$ is called a \textbf{weight measure} if it satisfies
\begin{equation*}
	\int_{x \in X} e^{- \mathsf{d} (x,y)} \mathrm{d}\mu_{\mathrm{w}} = 1
\end{equation*}
for any $y \in X$. 
\end{definition}

In the literature (e.g.~\cite{Mec13,LeiWil}), $\mu_{\mathrm{w}}$ is often allowed to be a signed measure, but we only consider positive Radon measures in what follows just to avoid further technicalities. The weight measure is not unique, but it does not matter as proved in \cite[\S 2.1]{Willerton}.

\begin{remark}
	We observe $\mu_{\mathrm{w}}(X) >1$ when $X$ contains more than two points. In particular, the weight measure cannot be a probability measure for any non-trivial metric space.
\end{remark}

We compute the $(\mu , \Gamma^{\mathrm{triv}})$-magnitude when we take $\mu$ to be a weight measure. The result is as follows.

\begin{theorem} \label{thmgwgtms}
	Let $(X , \mathsf{d})$ be a geodesic metric space which admits a weight measure $\mu_{\mathrm{w}}$, and assume that it puts no mass on non-proper chains and the total volume $\mu_{\mathrm{w}} (X)$ is finite.  Then, we can compute each partial magnitude explicitly as
\begin{equation*}
		\mathrm{Mag} (X , \mathsf{d} , \mu_{\mathrm{w}}, \Gamma^{\mathrm{triv}}; N) = \begin{cases}
			\mu_{\mathrm{w}} (X) &\quad \text{ if $N$ is even} \\
			0 &\quad \text{ if $N$ is odd .}
		\end{cases}
\end{equation*}
In particular, we have
\begin{equation*}
	\lim_{N \to \infty} \mathrm{Mag} (X , \mathsf{d} , \mu_{\mathrm{w}}, \Gamma^{\mathrm{triv}} ; 2N) = \mu_{\mathrm{w}} (X).
\end{equation*}
\end{theorem}

\begin{proof}
	We use the formula in Theorem \ref{thgmsnmnpc}. We compute, for each $n \ge 1$,
	\begin{align*}
		\int_{{X}^{n+1}}  e^{-\sum_{k=1}^n \mathsf{d}(x_{k}, x_{k+1})} \mathrm{d} \mu_{\mathrm{w}}^{n+1} &= \int_X \mathrm{d} \mu_{\mathrm{w}} \prod_{k=1}^n \int_X e^{ - \mathsf{d} (x_{k-1} , x_k) } \mathrm{d} \mu_{\mathrm{w}} \\
		&= \mu_{\mathrm{w}} (X)
	\end{align*}
	by the definition of the weight measure. 
\end{proof}

Leinster--Roff \cite[Corollary 6.4]{LR21} proved the existence of a \textbf{balanced measure} on any non-empty compact metric space, which is closely related to the weight measure \cite[Lemma 5.6]{LR21}. The defining property of the balanced measure is that
\begin{equation} \label{eqdfblmsb}
	\int_{x \in X} e^{- \mathsf{d} (x,y)} \mathrm{d} \mu_{\mathrm{b}} = c_{\mathrm{b}}
\end{equation}
holds for some constant $c_{\mathrm{b}} >0$ for any $y \in \mathrm{supp} (\mu_{\mathrm{b}})$. When $\mu_{\mathrm{b}}$ is a probability measure whose support consists of at least two points, we have $0 < c_{\mathrm{b}} < 1$. In this case, exactly the same argument as in the proof of Theorem \ref{thmgwgtms} immediately implies the following.

\begin{corollary} \label{thmgbgtms}
	Let $(X , \mathsf{d})$ be a geodesic metric space which admits a balanced probability measure $\mu_{\mathrm{b}}$, satisfying (\ref{eqdfblmsb}), and assume that it puts no mass on non-proper chains. Then, we have
\begin{equation*}
		\mathrm{Mag} (X , \mathsf{d} , \mu_{\mathrm{b}}, \Gamma^{\mathrm{triv}}) =\frac{1}{1+c_{\mathrm{b}}}.
\end{equation*}
\end{corollary}

\begin{proof}
	Arguing as in the proof of Theorem \ref{thmgwgtms}, for each $n \ge 1$ we have
	\begin{align*}
		\int_{{X}^{n+1}}  e^{-\sum_{k=1}^n \mathsf{d}(x_{k}, x_{k+1})} \mathrm{d} \mu_{\mathrm{b}}^{n+1} &= \int_{\mathrm{supp} (\mu_{\mathrm{b}})^{n+1}} e^{ -  \sum_{k=1}^n \mathsf{d} (x_{k-1} , x_k)} \mathrm{d} \mu_{\mathrm{b}}^{n+1} \\
		&= \int_{\mathrm{supp} (\mu_{\mathrm{b}})} \mathrm{d} \mu_{\mathrm{b}} \prod_{k=1}^n \int_{\mathrm{supp} (\mu_{\mathrm{b}})} e^{ - \mathsf{d} (x_{k-1} , x_k) } \mathrm{d} \mu_{\mathrm{b}} \\
		&= c_{\mathrm{b}}^{n} .
	\end{align*}
	Since $0<c_{\mathrm{b}}<1$, we have
	\begin{equation*}
		\mathrm{Mag} (X , \mathsf{d} , \mu_{\mathrm{b}}, \Gamma^{\mathrm{triv}}) = 1+\sum_{n=1}^{\infty} (-1)^n c_{\mathrm{b}}^{n} = \frac{1}{1+c_{\mathrm{b}}}
	\end{equation*}
	as required.
\end{proof}

Computation as above also gives the following result, which guarantees the convergence (as opposed to the convergence of even partial sums) of the magnitude by introducing an auxiliary regularisation parameter $0 < c < 1$.

\begin{corollary} \label{crmwmrgpc}
	Let $(X , \mathsf{d})$ be a geodesic metric space which admits a weight measure $\mu_{\mathrm{w}}$, and assume that it puts no mass on non-proper chains and the total volume $\mu_{\mathrm{w}} (X)$ is finite.  Then, for any $0 < c < 1$, its $(c \cdot \mu_{\mathrm{w}} , \Gamma^{\mathrm{triv}})$-magnitude exists and can be computed as 
	\begin{equation*}
		\mathrm{Mag} (X , \mathsf{d} , c \cdot \mu_{\mathrm{w}}, \Gamma^{\mathrm{triv}}) = \frac{c}{1+c} \mu_{\mathrm{w}} (X) .
	\end{equation*}
\end{corollary}

\begin{proof}
Arguing as in the proof of Corollary \ref{thmgbgtms}, we compute, for each $n \ge 1$,
	\begin{align*}
		\int_{{X}^{n+1}}  e^{-\sum_{k=1}^n \mathsf{d}(x_{k}, x_{k+1})} \mathrm{d} (c \cdot \mu_{\mathrm{w}})^{n+1} &= c^{n+1} \int_{X^{n+1}} e^{ -  \sum_{k=1}^n \mathsf{d} (x_{k-1} , x_k)} \mathrm{d} \mu_{\mathrm{w}}^{n+1} \\
		&=c^{n+1} \int_X \mathrm{d} \mu_{\mathrm{w}} \prod_{k=1}^n \int_X e^{ - \mathsf{d} (x_{k-1} , x_k) } \mathrm{d} \mu_{\mathrm{w}} \\
		&= c^{n+1} \mu_{\mathrm{w}} (X)
	\end{align*}
	by the definition of the weight measure. We thus get
	\begin{equation*}
		\mathrm{Mag} (X , \mathsf{d} ,c \cdot \mu_{\mathrm{w}}, \Gamma^{\mathrm{triv}}) = \sum_{n=0}^{\infty} (-1)^n c^{n+1} \mu_{\mathrm{w}} (X) = \frac{c}{1+c} \mu_{\mathrm{w}} (X)
	\end{equation*}
	as required.
\end{proof}

\subsection{Homogeneous magnitude theorem}

A compact homogeneous Riemannian manifold $(X,g)$ admits a left-invariant volume form $\mathrm{dvol}_g$ which defines a measure of finite volume. Then it is well-known (called the Speyer formula in \cite[Theorem 1]{LeiWil}) that
\begin{equation*}
	\mu_{\mathrm{w} , g}:= \frac{\mathrm{dvol}_g}{\int_{x \in X} e^{- \mathsf{d}_g (x,y) }\mathrm{dvol}_g}
\end{equation*}
is a weight measure, where the denominator does not depend on $y \in X$ because of the transitive action of the isometry group. Since $\mu_{\mathrm{w},g}$ is absolutely continuous with respect to the Lebesgue measure, it puts no mass on non-proper chains. Thus, we can apply Theorem \ref{thmgwgtms} and Corollary \ref{thmgbgtms} to immediately get the following result, analogously to what is called Speyer's Homogeneous Magnitude Theorem in \cite[Theorem 1]{Willerton}.

\begin{corollary} \label{crmgwgtms}
	Let $X$ be a compact homogeneous Riemannian manifold endowed with a left-invariant metric $g$, and the associated distance function $\mathsf{d}_g$ and the weight measure $\mu_{\mathrm{w} , g}$. Then, we have
\begin{equation*}
		\mathrm{Mag} (X , \mathsf{d}_g , \mu_{\mathrm{w,g}}, \Gamma^{\mathrm{triv}}; N) = \begin{cases}
			\mu_{\mathrm{w},g} (X)  &\quad \text{ if $N$ is even} \\
			0 &\quad \text{ if $N$ is odd .}
		\end{cases}
\end{equation*}
where
\begin{equation*}
	\mu_{\mathrm{w},g} (X) = \frac{\mathrm{Vol}_{g} (X)}{\int_{x \in X} e^{- \mathsf{d}_g (x,y)}\mathrm{dvol}_g }.
\end{equation*}
	Moreover, for any $0 < c < 1$, we have
\begin{equation*}
	\mathrm{Mag} (X , \mathsf{d}_g , c \cdot \mu_{\mathrm{w} , g}, \Gamma^{\mathrm{triv}}) = \frac{c}{1+c} \mu_{\mathrm{w} , g} (X).
\end{equation*}
\end{corollary}

Willerton \cite[\S 5]{Willerton} computed the asymptotic expansion of the volume of the weight measure as $t \to \infty$ when $(X,g)$ is a compact homogeneous Riemannian manifold. In the notation of this paper, it equals
\begin{equation*}
	\mathrm{Mag} (X , \mathsf{d}_{t^2g} , \mu_{\mathrm{w},t^2g}, \Gamma^{\mathrm{triv}}) = \mathrm{Mag} (X , t \mathsf{d}_g , t^{\dim X} \mu_{\mathrm{w},g}, \Gamma^{\mathrm{triv}}) =  \frac{\mathrm{Vol}_{g} (X)}{\int_{x \in X} e^{- t \mathsf{d}_g (x,y)}\mathrm{dvol}_g }.
\end{equation*}
It is important to note that both the metric and the measure depend on $t$ on the left hand side, and that it is different from $\mathrm{Mag} (X , t \mathsf{d}_g , \mu_{\mathrm{w},g}, \Gamma^{\mathrm{triv}})$. For example, when $X$ is a sphere of radius $R$ with the standard round metric, we have $\mu_{\mathrm{w} , g}= \mathrm{dvol}_g / (1-e^{- \pi R})$ which implies
\begin{equation*}
	\mathrm{Mag} (S^1_R , t \mathsf{d}_g , t \mu_{\mathrm{w},g}, \Gamma^{\mathrm{triv}}) = \frac{2 \pi R}{1-e^{- t\pi R}}
\end{equation*}
but the computation of $\mathrm{Mag} (S^1_R , t \mathsf{d}_g , \mu_{\mathrm{w},g}, \Gamma^{\mathrm{triv}})$ seems very difficult, as it is essentially the same as in \S \ref{scexmfdcl} by using $\mu_{\mathrm{w} , g}= \mathrm{dvol}_g / (1-e^{- \pi R})$.

\begin{remark} \label{rmscmtvf}
	For a compact homogeneous Riemannian manifold $(X,g)$, the weight measure is by definition a constant multiple of the volume form, where the constant is given by
	\begin{equation*}
		\tilde{c}:= \frac{1}{\int_{x \in X} e^{- \mathsf{d}_g (x,y) }\mathrm{dvol}_g } = \frac{\mathrm{Vol}_{g} (X)}{\int_{(x,y) \in X^2} e^{- \mathsf{d}_g (x,y) }\mathrm{dvol}^2_g} 
	\end{equation*}
	for any $y \in X$. This observation implies that the $(\mu , \Gamma^{\mathrm{triv}})$-magnitude depends highly non-trivially on scaling the measure: $\mathrm{Mag} (X , \mathsf{d}_g , \mu_{\mathrm{w} , g}, \Gamma^{\mathrm{triv}}) = \mathrm{Mag} (X , \mathsf{d}_g , \tilde{c} \cdot \mathrm{dvol}_g , \Gamma^{\mathrm{triv}})$ is computable but $\mathrm{Mag} (X , \mathsf{d}_g , \mathrm{dvol}_g , \Gamma^{\mathrm{triv}})$ is difficult to compute, even for $S^1$ as we saw above.
\end{remark}

\subsection{Positive definite metric spaces}

Compact positive definite metric spaces, i.e.~compact metric spaces such that for any finite subset $A$ the matrix $Z_A$ in (\ref{eqdfsmlmx}) is positive definite, play an important role in the study of magnitude. We recall its fundamental properties from \cite[\S 2]{Mec13}, where all details can be found. For a compact metric space $(X , \mathsf{d})$, we define a bilinear form
\begin{equation*}
	\mathcal{Z}_X ( \mu_1 , \mu_2 ) := \int_{x \in X} \int_{y \in X} e^{-\mathsf{d} (x,y)} \mathrm{d} \mu_1 \mathrm{d} \mu_2
\end{equation*}
on the space of finite signed Borel measures on $X$. Meckes \cite[Lemma 2.2]{Mec13} proved that $(X , \mathsf{d})$ is positive definite if and only if $\mathcal{Z}_X$ is positive definite, and that there exists a finite positive Borel measure $\mu_{\mathrm{w}}$ on a compact positive definite space $(X , \mathsf{d})$ such that
\begin{equation*}
	\mu_{\mathrm{w}}(X)= \sup_{\mu} \frac{\mu(X)^2}{\mathcal{Z}_X ( \mu , \mu )} 
\end{equation*}
holds \cite[Proposition 2.9]{Mec13}, where the supremum is over all non-zero finite positive Borel measures. Under the extra assumption that $(X , \mathsf{d})$ is positively weighted \cite[page 744]{Mec13}, there exists a (positive, not signed) weight measure \cite[Corollary 2.10]{Mec13} whose volume agrees with the classical magnitude (\ref{eqdfclmg}) by \cite[Theorems 2.3 and 2.4]{Mec13}. This quantity is even continuous with respect to the Gromov--Hausdorff distance on the class of compact positive definite metric spaces \cite[Proposition 2.11]{Mec13}, and that it agrees with the limit of the magnitude of finite subsets in $X$ converging to $(X , \mathsf{d})$ in the Gromov--Hausdorff distance \cite[Corollary 2.7]{Mec13}. Important classes of positively weighted positive definite spaces are as given in \cite[Theorem 3.6]{Mec13} and \cite[Lemma 2.8]{Mec13}. For finite positive definite spaces, the magnitude can be computed as in \cite[Proposition 2.4.3]{Lein13}. It is also known that a compact Riemannian manifold with property (MR) is positive definite once the distance is re-scaled \cite[Theorem 2.1]{GGL22}, where the property (MR) is as defined in \cite[Definition 3.3]{GGL24}.

Thus the weight measure, if exists, plays a hugely important role in the theory of magnitude. Whether this measure puts no mass on non-proper chains is a subtle problem (indeed not satisfied for line segments, as we see in \S \ref{sccbism}), but we get the following corollary when it does, by combining Theorem \ref{thmgwgtms} and the results in \cite[\S 2]{Mec13} by Meckes mentioned above.

\begin{corollary} \label{clmwsmmz}
	Let $(X , \mathsf{d})$ be a positively weighted compact positive definite geodesic metric space with the weight measure $\mu_{\mathrm{w}}$ that puts no mass on non-proper chains. Then, we have
	\begin{equation*}
		\lim_{N \to \infty} \mathrm{Mag} (X , \mathsf{d} , \mu_{\mathrm{w}}, \Gamma^{\mathrm{triv}} ; 2N) = \sup_{Z \subset X, \; \textup{finite}} \mathrm{Mag} (Z , \mathsf{d} |_Z) , 
	\end{equation*}
	and also
	\begin{equation*}
		\lim_{N \to \infty} \mathrm{Mag} (X , \mathsf{d} , \mu_{\mathrm{w}}, \Gamma^{\mathrm{triv}} ; 2N)  = \sup_{\mu} \frac{\mu(X)^2}{\mathcal{Z}_X ( \mu , \mu )}  ,
		\end{equation*}
		where the supremum is over all non-zero finite positive measures. Moreover, the $(\mu_{\mathrm{w}}, \Gamma^{\mathrm{triv}})$-magnitude as above is continuous with respect to the Gromov--Hausdorff distance on the class of compact positive definite metric spaces with the scaled weight measure.
\end{corollary}

It is known that the above assumptions are satisfied for round spheres \cite[Lemma 2.8 (2) and Theorem 3.6 (6)]{Mec13}. Combined with Theorem \ref{thmfkeqm} and Corollary \ref{crmgwgtms}, we have the following result which exhibits a relationship between the classical magnitude (\ref{eqdfclmg}) and the magnitude of Fekete configurations for $\mathbb{CP}^1$, which is the only K\"ahler manifold that can be realised as a sphere.

\begin{corollary} \label{thmclmfkc1}
	For $\mathbb{CP}^1$ with the Fubini--Study metric $\mathsf{d}$, the classical magnitude can be computed in terms of the re-scaled limit of the magnitude of Fekete configurations. More precisely, for a Fekete configuration $\mathcal{F}_m$ in $(\mathbb{CP}^1, \mathsf{d})$ and a constant $c >1$ defined by
	\begin{equation*}
		c  := \frac{2( 1+r^2)}{1+e^{- \pi r}}
	\end{equation*}
	where we set the diameter of $\mathbb{CP}^1$ with respect to $\mathsf{d}$ to be $\pi r$, we have
	\begin{equation*}
		\sup_{Z \subset \mathbb{CP}^1, \; \textup{finite}} \mathrm{Mag} (Z , \mathsf{d} |_Z ) = \lim_{\gamma \nearrow c} \lim_{m \to \infty} \frac{2 \gamma}{m} \mathrm{Mag} ( \mathcal{F}_m , \mathsf{d}_{\log (m / \gamma)} )
	\end{equation*}
	and also
	\begin{equation*}
		\sup_{Z \subset \mathbb{CP}^1, \; \textup{finite}} \mathrm{Mag} (Z , \mathsf{d} |_Z ) = \lim_{m \to \infty} \frac{c}{m} \mathrm{Mag} ( \mathcal{F}_m , \mathsf{d}_{\log (m / c)} , \mu_{\sharp}, \Gamma^{\mathrm{triv}} ; 2N)
	\end{equation*}
	uniformly for all $N \in \mathbb{Z}_{>0}$.
\end{corollary}

\begin{proof}
	In terms of the notation in Theorem \ref{thmfkeqm}, $X =K = \mathbb{CP}^1$, $\mathsf{d}$ is the distance induced from the Fubini--Study metric, and $\phi_{\mathrm{psh}}$ is the K\"ahler potential of the Fubini--Study metric. In this case, the equilibrium measure is given by the Fubini--Study volume form divided by its total volume \cite[(1.3) and (1.4)]{BB10}.
	
	Let $\mu_{\mathrm{FS}}$ be the volume form of the Fubini--Study metric on $\mathbb{CP}^1$, which agrees with the round metric under the identification $\mathbb{CP}^1 \cong S^2$. We further normalise $\mu_{\mathrm{FS}}$ to be the probability measure, so that it agrees with the equilibrium measure. The weight measure $\mu_{\mathrm{w}}$ on $\mathbb{CP}^1$ is given by $c \cdot \mu_{\mathrm{FS}}$ as in Remark \ref{rmscmtvf}, with the constant $c$ given by
	\begin{equation*}
		c = \left( \int_{\mathbb{CP}^1} e^{- \mathsf{d} (x,y)} \mathrm{d} \mu_{\mathrm{FS}} (y) \right)^{-1} = \frac{(4 \pi r^2)^2}{2 (2 \pi)^2 r^4}  \frac{1+r^2}{1+ e^{-e \pi r}} =  \frac{2(1+r^2)}{1+e^{- \pi r}} >1,
	\end{equation*}
	where we repeated the computation in \S \ref{ssc2shn2ex}.
	
	Pick any $0 < \gamma < c$ and set $b:= \gamma /c \in (0,1)$. Then, by Proposition \ref{ppscmdc}, for any finite subset $Z \subset \mathbb{CP}^1$ and for each $N' \in \mathbb{Z}_{>0}$, we have
	\begin{align}
	\mathrm{Mag} ( \mathbb{CP}^1 , \mathsf{d}, b \cdot \mu_{\mathrm{w}}, \Gamma^{\mathrm{triv}} ; N') &= \mathrm{Mag} ( \mathbb{CP}^1 , \mathsf{d} , \gamma \cdot \mu_{\mathrm{FS}}, \Gamma^{\mathrm{triv}} ; N') \notag \\
	&= \lim_{m \to \infty}  \mathrm{Mag} ( \mathcal{F}_m , \mathsf{d} |_{\mathcal{F}_m} , \gamma \cdot \mu_{\mathcal{F}_m , \mathrm{em}}, \Gamma^{\mathrm{triv}} ; N') \notag \\
	&= \lim_{m \to \infty} \frac{\gamma}{m} \mathrm{Mag} ( \mathcal{F}_m , \mathsf{d}_{\log (m / \gamma)} , \mu_{\sharp}, \Gamma^{\mathrm{triv}} ; N') \label{eqthcp1fkc}
	\end{align}
	where we used the weak convergence of $\mu_{\mathcal{F}_m , \mathrm{em}}$ to the equilibrium measure $\mu_{\mathrm{FS}}$ in the second equality (see the proof of Theorem \ref{thmfkeqm}), noting that the sum is finite here. We also note that $\log (m / \gamma) > 0$ for all large enough $m$, so that $\mathsf{d}_{\log (m / \gamma)}$ is well-defined. Since there are only finitely many terms involved in the sum, the argument above also works for $\gamma = c$, which gives the second claim by setting $N' = 2N$ and by recalling Corollary \ref{clmwsmmz} (see also Theorem \ref{thmgwgtms}).

	We prove the first claim. The limit $N' \to \infty$ of the left hand side of (\ref{eqthcp1fkc}) can be computed as in the proof of Corollary \ref{crmwmrgpc}. Indeed, we have
	\begin{equation*}
		\mathrm{Mag} ( \mathbb{CP}^1 , \mathsf{d}, b \cdot \mu_{\mathrm{w}}, \Gamma^{\mathrm{triv}} ; N') = b \frac{1-(-b)^{N'+1}}{1+b}\mu_{\mathrm{w}} (\mathbb{CP}^1)
	\end{equation*}
	which converges to $\mu_{\mathrm{w}} (\mathbb{CP}^1) /2$ when we take the limit $N' \to \infty$, and then $b \nearrow 1$. We thus get
	\begin{equation} \label{eqthcp1fkccl}
		\lim_{b \nearrow 1} \lim_{N' \to \infty} \mathrm{Mag} ( \mathbb{CP}^1 , \mathsf{d}, b \cdot \mu_{\mathrm{w}}, \Gamma^{\mathrm{triv}} ; N') = \frac{1}{2} \mu_{\mathrm{w}} (\mathbb{CP}^1) = \frac{1}{2}\sup_{Z \subset \mathbb{CP}^1, \; \textup{finite}} \mathrm{Mag} (Z , \mathsf{d} |_Z )
	\end{equation}
	by Corollary \ref{clmwsmmz}.
	
	We now take the limit $N' \to \infty$, and then $\gamma \nearrow c$, on the right hand side of (\ref{eqthcp1fkc}), noting that the limits exist as above. For each fixed $0 < \gamma < c$, the limit $N' \to \infty$ is uniform in $m$, since for any $x \in \mathbb{CP}^1$ we have
	\begin{equation*}
		\gamma \int_{y \in \mathbb{CP}^1} e^{- \mathsf{d} (x,y)} \mathrm{d} \mu_{\mathcal{F}_m , \mathrm{em}} < c \int_{y \in \mathbb{CP}^1} e^{- \mathsf{d} (x,y)} \mathrm{d} \mu_{\mathcal{F}_m , \mathrm{em}} \to \int_{y \in \mathbb{CP}^1} e^{- \mathsf{d} (x,y)} \mathrm{d} \mu_{\mathrm{w}} =1
	\end{equation*}
	as $m \to \infty$, meaning that the maximum over $\mathbb{CP}^1$ of the left hand side above is less than 1 for all large enough $m$ once $\gamma$ is fixed. We then argue as in the proof of Theorem \ref{thmfkeqm} to exchange the limits $m \to \infty$ and $N' \to \infty$. Combined with the previous limit (\ref{eqthcp1fkccl}), we get
	\begin{equation*}
		\sup_{Z \subset \mathbb{CP}^1, \; \textup{finite}} \mathrm{Mag} (Z , \mathsf{d} |_Z ) = \lim_{\gamma \nearrow c} \lim_{m \to \infty} \frac{2 \gamma}{m} \lim_{N' \to \infty} \mathrm{Mag} ( \mathcal{F}_m , \mathsf{d}_{\log (m / \gamma)} , \mu_{\sharp}, \Gamma^{\mathrm{triv}} ; N')
	\end{equation*}
	which gives the required result.
\end{proof}

\begin{remark}
We observe that the bilinear form $\mathcal{Z}_X$ is closely related to our definition of the magnitude, since
\begin{equation*}
	\mathcal{Z}_X ( \mu , \mu ) = \int_{x \in X} \int_{y \in X} e^{-\mathsf{d} (x,y) } \mathrm{d} \mu \mathrm{d} \mu
\end{equation*}
is exactly the $n=1$ term of the magnitude in Theorem \ref{thmggrmmf}, as long as $\mu$ puts no mass on $X^{2} \setminus P_1 (X)$. We can equivalently write
\begin{equation} \label{eqzxblffpm}
	\mathcal{Z}_X ( \mu , \mu ) = \mu (X) - \mathrm{Mag} (X , \mathsf{d} , \mu;1).
\end{equation}
Thus, as long as $\mathrm{Mag} (X , \mathsf{d} , \mu;1) \neq \mu (X)$, we can write
\begin{equation*}
	\frac{\mu (X)^2}{\mathcal{Z}_X ( \mu , \mu )} = \frac{\mu (X)^2}{\mu (X) - \mathrm{Mag} (X , \mathsf{d} , \mu;1)}.
\end{equation*}
When $(X , \mathsf{d})$ is positive definite, the supremum of the above quantity over all non-zero measures is called the maximum diversity of $(X , \mathsf{d})$ by Meckes \cite[(2.2)]{Mec13}, and discussed in detail in \cite{LR21} and \cite[Chapter 6]{Leinbook}. After re-scaling, we may take the supremum of $\mu (X)^2 / \mathcal{Z}_X ( \mu , \mu )$ over all probability measures. If there exists a Borel probability measure $\mu_{\mathrm{max}}$ which attains
\begin{equation*}
	D_{\mathrm{max}} (X) := \sup_{\mu} \frac{1}{\mathcal{Z}_X ( \mu , \mu )} = \frac{1}{\mathcal{Z}_X ( \mu_{\mathrm{max}} , \mu_{\mathrm{max}} )}
\end{equation*}
and puts no mass on $X^{2} \setminus P_1 (X)$, we can write
\begin{equation*}
	D_{\mathrm{max}} (X) = \frac{1}{1 - \mathrm{Mag} (X , \mathsf{d} , \mu_{\mathrm{max}};1)}.
\end{equation*}

It would be nice if we can say that the maximum diversity is achieved for a measure which maximises the first partial $\mu$-magnitude, but (\ref{eqzxblffpm}) only holds for the measure which puts no mass on $X^{2} \setminus P_1 (X)$ and hence the range of supremum becomes subtle.
\end{remark}

\subsection{Closed bounded intervals with weight measure} \label{sccbism}

When $(X , \mathsf{d})$ is a subset of $\mathbb{R}$, with the restriction of the standard metric, it admits a weight measure \cite[Lemma 2.8]{Mec13}. Various examples were computed in e.g.~\cite[\S 3.2]{Lein13}, \cite[\S 2, 3]{LeiWil}, \cite[\S 2.2]{Willerton}. In this case, however, the weight measure can have a singular component which puts mass on non-proper chains. For example, the weight measure for the closed interval $[a,b]$ of length $L:=b-a$ is given by $\mu_{\mathrm{w}} = (\delta_a + \delta_b + \mu )/2$ where $\mu$ is the standard Lebesgue measure \cite[Theorem 2]{Willerton}.

We now set $X$ to be the closed interval $[a,b]$ of length $L$, and compute the $\mu_{\mathrm{w}}$-magnitude of this space. Any two points in $X$ can be joined by a unique length-minimising geodesic. We identify subsets of $X^{n+1} \setminus P_n (X)$ that has non-trivial measure with respect to $\mu_{\mathrm{w}}$. Since the only singular component of $\mu_{\mathrm{w}}$ is the delta measure at the boundary, the elements that contribute non-trivially to the integral are necessarily of the form
\begin{equation*}
	\{ (x_0 , \dots , x_n) \in X^{n+1} \mid x_{k-1}=x_{k}= \text{$a$ or $b$ for some $k=1 , \dots , n$ } \} .
\end{equation*}
We list all possible patterns of indices that arise in $X^{n+1} \setminus P_n (X)$, which are shown to correspond one-to-one with ordered partitions of $\tilde{n}:=n+1$ apart from $1 + 1 + \cdots + 1$, as follows.

For example, when $n=3$, the ordered partition $2+1+1$ of $\tilde{n}=4$ corresponds to the chains with $x_0=x_1 \neq x_2 \neq x_3$, or more explicitly
\begin{equation*}
	\{ (x_0 , x_1, x_2 , x_3) \in X^{4} \mid x_0=x_1 = a, x_2=b, x_3=a, \text{ or } x_0=x_1 = b, x_2=a, x_3=b \} .
\end{equation*}
Thus, parts greater than or equal to $2$ of the partition correspond to the number of repeated variables in $X^{n+1} \setminus P_n (X)$. We write
\begin{equation*}
	\mathcal{C}_{n+1} := \{ \text{ordered partitions of $\tilde{n}$} \} \setminus \{ 1 + 1 + \cdots + 1 \} .
\end{equation*}
To simplify the terminology, for each $\lambda \in \mathcal{C}_{n+1}$, a part greater than or equal to $2$ will be called \textbf{repeated indices}. We consider repeated indices to be \textbf{non-adjacent} if they are separated by $1$, and a \textbf{cluster} is meant to be a group of adjacent repeated indices. For example, if $\tilde{n}=8$ and $\lambda = 3+2+1+2$, there are three groups of repeated indices (namely $3$, $2$ and $2$) and two clusters ($3+2$ and $2$). For each $\lambda \in \mathcal{C}_{n+1}$, we define
\begin{equation*}
	f ( \lambda ) := \text{the number of clusters in } \lambda ,
\end{equation*}
and
\begin{equation*}
	g ( \lambda ) := \text{the total number of plus signs within clusters in } \lambda ,
\end{equation*}
For example, if $\tilde{n}=14$ and $\lambda = 3+2+1+2+ 1 + 2+ 3$, then $f(\lambda)=3$ ($3+2$, $2$, and $2+3$) and $g(\lambda)=2$ (one $+$ from $3+2$, and another from $2+3$); if $\tilde{n}=14$ and $\lambda' =3+ 2+2+1+ 1+2+3$, we have $f(\lambda')=2$ ($3+2+2$ and $2+3$) and $g(\lambda')=3$ (two from $3+2+2$ and one from $2+3$).

We compute the integrals for the magnitude, first when there is a single group of $i$ repeated variables $x_k = \cdots = x_{k+i}$, which corresponds to the partition $\lambda = 1 + 1 + \cdots + i + \cdots + 1$ which has $f(\lambda) = 1$ and $g(\lambda) = 0$. Since $\mu_{\mathrm{w}}$ has no singular component in $(a,b)$, the locus $\{ x_k = \cdots = x_{k+i} \} $ has measure zero inside $(a,b)^i$ with respect to $\mu^i_{\mathrm{w}}$. Thus the only non-trivial contributions from repeated variables come from the boundary $\partial X = \{ a \} \cup \{ b \}$. Hence we get
\begin{align*}
	&\int_{\bm{x} \in X^{n+1} \text{ with } x_k = \cdots = x_{k+i}}  e^{-  \mathrm{Len} (\gamma (\bm{x})) } \mathrm{d} \mu_{\mathrm{w}}^{n+1} \\
	&=\sum_{y=a,b} \int_{X^k} e^{- \mathsf{d}(x_0, x_1) - \cdots - \mathsf{d} (x_{k-1},y)}\mathrm{d} \mu^k_{\mathrm{w}} \times \int_{X^{n-k-i}} e^{-  \mathsf{d}(y, x_{k+i+1}) - \cdots - \mathsf{d} (x_{n-1},x_n) }\mathrm{d} \mu^{n-k-i}_{\mathrm{w}}	\\
	&=2 \\
	&= | \partial X |^{f(\lambda)} e^{-  L g(\lambda) }
\end{align*}
by recalling that $\mu_{\mathrm{w}}$ is the weight measure of $(X , \mathsf{d})$. Computation is entirely the same when there are multiple non-adjacent groups, but there is an extra complication coming from adjacent groups of repeated indices. For example, if $\tilde{n}=4$ with $x_0=x_1 \neq x_2=x_3$, which corresponds to the partition $\lambda = 2+2$ with $f(\lambda)=1$ and $g(\lambda)=1$, we have
\begin{align*}
	\int_{\bm{x} \in X^{4} \text{ with } x_0=x_1 \neq x_2=x_3}  e^{-  \mathsf{d}(x_0, x_1) - \mathsf{d}(x_1, x_2) - \mathsf{d}(x_2, x_3) } \mathrm{d} \mu^4_{\mathrm{w}} &=e^{-  \mathsf{d}(a,b) } + e^{- \mathsf{d}(b,a) } \\
	&=| \partial X |^{f(\lambda)}e^{-  L g(\lambda) }.
\end{align*}
Writing
\begin{equation*}
	Q_n (X; \lambda ) := \{ \bm{x} = (x_0, \dots, x_n ) \in X^{n+1} \mid \text{the repeated indices are as given by $\lambda \in \mathcal{C}_{n+1}$} \},
\end{equation*}
for the set of non-proper chains with the repetition of indices specified by $\lambda$, the above arguments imply that the integrals in the magnitude can be computed as
\begin{align*}
	&\int_{\bm{x} \in P_n (X)}  e^{- \mathrm{Len} (\gamma (\bm{x})) } \mathrm{d} \mu^{n+1}_{\mathrm{w}} \\
	&= \int_{\bm{x} \in {X}^{n+1}}  e^{-  \mathrm{Len} (\gamma (\bm{x})) } \mathrm{d} \mu^{n+1}_{\mathrm{w}}  - \sum_{\lambda \in \mathcal{C}_{n+1}}  \int_{\bm{x} \in Q_n (X; \lambda )}  e^{- \mathrm{Len} (\gamma (\bm{x})) } \mathrm{d} \mu^{n+1}_{\mathrm{w}}  \\
	&= \int_{\bm{x} \in {X}^{n+1}}  e^{- \mathrm{Len} (\gamma (\bm{x})) } \mathrm{d} \mu^{n+1}_{\mathrm{w}}  - \sum_{\lambda \in \mathcal{C}_{n+1}} | \partial X |^{f(\lambda)} e^{-Lg(\lambda)}
\end{align*}
where in the second integral repeated variables occur for indices specified by $\lambda$. We thus get
\begin{equation*}
	\mathrm{Mag} (X , \mathsf{d} , \mu_{\mathrm{w}}; N) = \mu_{\mathrm{w}} (X) - \sum_{n=1}^{N} (-1)^{n} \sum_{\lambda \in \mathcal{C}_{n+1}} |\partial X |^{f(\lambda)}e^{-Lg(\lambda)} 
\end{equation*}
for any $N \in \mathbb{Z}_{>0}$. We recall $\mu_{\mathrm{w}} (X) = 1 + \frac{L}{2}$ \cite[Theorem 3.2.2, Proposition 3.2.1]{Lein13} and \cite[Theorem 2]{Willerton}. It is a non-trivial problem to determine whether this series converges as $N \to \infty$. If we allow ourselves to re-scale the measure by a constant $c>0$, we get
\begin{align*}
	\mathrm{Mag} (X , \mathsf{d} ,c \cdot \mu_{\mathrm{w}}) &= \frac{c}{1+c} \mu_{\mathrm{w}} (X) - \sum_{n=1}^{\infty} (-1)^{n}c^{n+1} \sum_{\lambda \in \mathcal{C}_{n+1}} |\partial X |^{f(\lambda)}e^{-Lg(\lambda)} \\
	&=\frac{c}{1+c} \left(1 + \frac{L}{2} \right) + \sum_{n=1}^{\infty} (-c)^{n+1} \sum_{\lambda \in \mathcal{C}_{n+1}} |\partial X |^{f(\lambda)} e^{-Lg(\lambda)}
\end{align*}
which is convergent for $0 < c < 1/4$, since $| \mathcal{C}_{n+1} | = 2^n -1$ and $f(\lambda) \le n$. Note that $g( \lambda )$ is counting the numbers when the indices flip from $a$ to $b$ or $b$ to $a$, which seems to be counting the number of geodesics between $a$ and $b$, just as we saw in \S \ref{sccbilm}.

\bibliography{magnitude.bib}

\begin{flushleft}
{\footnotesize
Department of Mathematics, Osaka Metropolitan University, \\
3-3-138, Sugimoto, Sumiyoshi-ku, Osaka, 558-8585, Japan.\\
Email: \texttt{yhashimoto@omu.ac.jp}}
\end{flushleft}\end{document}